\documentclass[12pt,a4paper,oneside]{amsart}
\usepackage[a4paper,left=3cm,right=3cm, top=3cm, bottom=3cm]{geometry}
\usepackage[latin1]{inputenc}
\usepackage{mathrsfs}
\usepackage[T1]{fontenc}
\usepackage{graphicx}
\usepackage{amsmath}
\usepackage{amsthm}
\usepackage{amssymb}
\usepackage{hyperref}
\usepackage{color}
\usepackage{enumerate}
\usepackage{esint}
\usepackage{textcomp}
\theoremstyle{plain}
\newtheorem{thm}{Theorem}[section]

\newtheorem{lem}[thm]{Lemma}
\newtheorem{prop}[thm]{Proposition}
\newtheorem*{propA}{Proposition A}
\theoremstyle{definition}
\newtheorem{defn}{Definition}[section]
\theoremstyle{remark}
\newtheorem{remark}{Remark}

\newcommand{\R}{\mathbb{{R}}}
\newcommand{\N}{\mathbb{{N}}}

\newcommand{\Lp}{L^p}

\newcommand{\Ho}{H_{0}^{1}(\Omega)}
\newcommand{\Hoo}{H_{0}^{2}(\Omega)}
\newcommand{\Omegabar}{\overline{\Omega}}

\newcommand{\Caratheodory}{\text{Carath\'{e}odory} }

\newcommand{\dOmega}{{\partial\Omega}}
\newcommand{\B}{\mathcal B}
\newcommand{\dB}{{\partial \B}}

\newcommand{\omegar}{\omega_{\bar r}}
\newcommand{\varphitilde}{\tilde\varphi_1}
\newcommand{\lambdatilde}{\tilde\lambda_1}
\newcommand{\intOmega}{\int_\Omega}

\newcommand{\xki}{x_{k,i}}
\newcommand{\XI}{x^{(i)}}
\newcommand{\muki}{\mu_{k,i}}
\newcommand{\Ox}{\mathcal{O}_x}
\newcommand{\Oz}{\mathcal{O}_z}
\newcommand{\OR}{\mathcal{R}}
\newcommand{\lambdabar}{\bar{\lambda}}

\begin{document}
	\title[Uniform bounds for higher-order problems]{Uniform bounds for higher-order semilinear problems in conformal dimension}
	
	\author[G.~Mancini]{Gabriele Mancini}
	\address[G.~Mancini]{Universit\`{a} Sapienza di Roma, Dipartimento di Scienze di Base e Applicate per l'Ingegneria, via Antonio Scarpa 16, 00161 Roma, Italy}
	\email{gabriele.mancini@uniroma1.it}
	
	\author[G.~Romani]{Giulio Romani}
	\address[G.~Romani]{Institut f\"{u}r Mathematik, Martin-Luther-Universit\"{a}t Halle-Wittenberg, 06099 Halle (Saale), Germany}
	\email{giulio.romani@mathematik.uni-halle.de}
	
	
\begin{abstract}
	\noindent We establish uniform a-priori estimates for solutions of the semilinear Dirichlet problem
	\begin{equation*}
		\begin{cases}
			(-\Delta)^mu=h(x,u)\quad&\mbox{in }\Omega,\\%
			u=\partial_nu=\cdots=\partial_n^{m-1}u=0\quad&\mbox{on }\dOmega,
		\end{cases}
	\end{equation*}
	where $h$ is a positive superlinear and subcritical nonlinearity in the sense of the Trudinger-Moser-Adams inequality, either when $\Omega$ is a ball or, provided an energy control on solutions is prescribed, when $\Omega$ is a smooth bounded domain.
	Our results are sharp within the class of distributional solutions. The analogous problem with Navier boundary conditions is also studied. Finally, as a consequence of our results, existence of a positive solution is shown by degree theory.
\end{abstract}
\maketitle

\section{Introduction and main results}
A-priori estimates play an important role in the theory of elliptic boundary value problems. In particular \textit{uniform} bounds, namely a-priori estimates in $L^\infty$ norm, are a key tool in order to show existence of solutions by degree theory.\\ 

Concerning second-order problems, the major contributions in this direction have been given by Brezis and Turner \cite{BT}, Gidas and Spruck \cite{GS} and de Figueiredo, Lions and Nussbaum \cite{dFLN}. They proved uniform bounds for positive classical solutions of the second-order semilinear Dirichlet problem 
\begin{equation}\label{DIRfsecondorder}
\begin{cases}
-\Delta u=h(x,u)\quad&\mbox{in }\Omega,\\%
u=0\quad&\mbox{on }\dOmega,
\end{cases}
\end{equation}
where $\Omega$ is a bounded smooth domain in $\R^N$, $N\geq2$. The main assumptions on the nonlinearity are positivity, superlinearity and a \textit{polynomial} subcritical growth at $\infty$, namely $0<f(x,s)<cs^p$ for some constant $c>0$ and $1<p<2^*-1$, where $2^*:=\frac{2N}{N-2}$ denotes the Sobolev critical exponent for $\Ho$. It is remarkable that the authors exploited different important techniques: Hardy-Sobolev inequalities in \cite{BT}, a blow-up method in \cite{GS}, a moving planes argument in \cite{dFLN}.

If we consider the problem \eqref{DIRfsecondorder} in the special case of $N=2$, then formally $2^*=\infty$ and one can replace the critical Sobolev growth $t\mapsto t^{2^*}$ with the \textit{exponential} map $t\mapsto e^{t^2}$, as the well-known Trudinger-Moser inequality shows:
\begin{align}\label{TM}
	\sup_{u\in H^1_0(\Omega),\;\|\nabla u\|_2\leq1}\int_\Omega e^{\beta|u|^2}dx\quad
	\begin{cases}
		\leq C|\Omega|,\quad&\mbox{if }\beta\leq 4\pi,\\
		=+\infty\quad&\mbox{if }\beta>4\pi.       
	\end{cases}
\end{align}
Hence one might expect that an uniform a-priori bound holds for nonlinearities up to this new critical growth. This question was investigated by Brezis and Merle in their seminal paper \cite{BM}. Indeed, as a consequence of a concentration-compactness result, the authors established local uniform bounds for solutions $(u_k)_k$ of the problems
\begin{equation}\label{BMeq}
\begin{cases}
-\Delta u=V_k(x)e^{u}\quad&\mbox{in}\,\Omega\subset\R^2,\\
u=0\quad&\mbox{on}\,\dOmega,
\end{cases}
\end{equation}
but the question about the existence of global a-priori bounds, considering for instance $0<a\leq V_k\leq b<+\infty$, even with the additional assumption $\|e^{u_k}\|_{L^1(\Omega)}\leq C$ for any $k$, was left open. However, at least for convex domains, the a-priori bound can be achieved thanks to some uniform boundary estimates in the already cited \cite{dFLN}; see also \cite{ChenLi}. In any case, as again the work of Brezis and Merle surprisingly shows, one may expect to find a-priori estimates for solutions in the uniform norm \textit{only} if the growth of the nonlinearity is at most $t\mapsto e^t$.
\vskip0.2truecm
If we consider now the higher-order counterpart of problem \eqref{DIRfsecondorder}, namely the Dirichlet problem
\begin{equation}\label{DIRf_generalpoly}
\begin{cases}
(-\Delta)^mu=h(x,u)\quad&\mbox{in }\Omega,\\%
u=\partial_nu=\cdots=\partial_n^{m-1}u=0\footnotemark\quad&\mbox{on }\dOmega,
\end{cases}
\end{equation}
\footnotetext{Here and in the sequel, the unit outer normal vector in $x$ is denoted by $n(x)$, thus $\partial_n u$ or alternatively $u_n$ stand for the normal derivative of $u$.}
\noindent \hspace{-0.43em}for $m\geq2$ and with a positive nonlinearity $h$ with power-growth at $\infty$, then the results in \cite{GS,dFLN} have been generalized to this context by Oswald \cite{Osw} and Soranzo \cite{Soranzo}, both dealing with positive radially symmetric solutions of \eqref{DIRf_generalpoly} for homogeneous nonlinearities in the ball. Some years later, Reichel and Weth \cite{RW} extended the theory admitting also nonpositive solutions, nonhomogeneous nonlinearities and general smooth domains, using a blow-up technique inspired by \cite{GS}. A growth condition in the spirit of \cite{GS} was required, namely that
\begin{equation}\label{GSgrowth}
\lim_{s\to\infty}\frac{h(x,s)}{s^p}=a(x)\qquad1<p<\tfrac{N+2m}{N-2m}=:2^*_{m}-1,\quad\,\mbox{uniformly in $\Omegabar$,}
\end{equation}
where $2^*_{m}$ is the critical Sobolev exponent for $H^m_0(\Omega)$ and $a>0$ in $\Omegabar$.\\
\indent In the special case of the conformal dimension $N=2m$, the well-known generalization of \eqref{TM} established by Adams \cite{Adams} for higher-order Sobolev spaces shows that the critical growth is still given by the map $t\mapsto e^{t^2}$. The fourth-order counterpart of the aforementioned concentration-compactness result in \cite{BM} was proved in \cite{ARS} highlighting also some new features, and then extended in the general context in \cite{Mart_cc}. In particular, this provides local uniform bounds in case of positive solution of the polyharmonic analogue of problem \eqref{BMeq}.
\vskip0.2truecm
Here, we want to complement these last results and investigate more deeply problem \eqref{DIRf_generalpoly} in the case of the conformal dimension, considering positive superlinear nonlinearities up to the map $t\mapsto e^t$. More precisely, we shall prove that if the nonlinearity $t\mapsto h(x,t)$ is controlled in a suitable sense by $t\mapsto e^{\gamma t}$ for some $\gamma>0$, then the a-priori bound holds for any solution of problem \eqref{DIRf_generalpoly}.

It is worth to mention that a related problem was also addressed in \cite{LRU}, where the operator involved is the $N$-laplacian in dimension $N$ and the nonlinearity either is growing less than $e^{t^\alpha}$ for some $\alpha\in(0,1)$, or behaves like $e^t$. The authors use Orlicz spaces techniques to cover the first alternative and some arguments inspired by Brezis and Merle for the second case.
Although in his Ph.D. Thesis \cite[Chapter 6]{Pass} Passalacqua improved the result, the gap between the growths $e^{t^\alpha}$ and $e^t$ was not completely filled: for instance, the growth $f(t)=e^t(1+t)^{-\alpha}$ with $\alpha>1$ was not allowed. Covering the remaining cases seems not attainable with such techniques. Let us also mention that a similar gap occurs also when dealing with coupled elliptic systems in critical dimension, see \cite{dFdOR}.

Here, instead, we manage to deal with all these nonlinearities applying a blow-up strategy inspired mainly by the work of Robert and Wei \cite{RobWei}. We may sketch our main argument as follows. Firstly we obtain that $\int_\Omega h(x,u)dx$ is uniformly bounded for any solution of the problem \eqref{DIRf_generalpoly}. Here the main obstruction is the lack of a maximum principle and good Green function estimates. Then, supposing the existence of an unbounded sequence of solutions, we appropriately rescale them and find a limit problem in the whole $\R^{2m}$. The $L^1$ control previously obtained yields a characterization of the limit profile and, furthermore, a contradiction with the concentration of the energy around each blow-up point.

\vskip0.3truecm

\noindent Throughout all the paper, the nonlinearity $h:\Omega\times\R\to\R^+$ is a \Caratheodory  function satisfying the following conditions:
\begin{enumerate}[H1)]
	\item $h\in L^\infty(\Omega\times[0,\tau])$ for all $\tau\in\R^+$; 
	\item there exist functions $f\in C^1([0,+\infty))$ satisfying assumption (\textit{A}) below and $0<a\in L^\infty(\Omega)\cap C(\Omega)$ such that $$\lim_{t\rightarrow+\infty}\frac{h(x,t)}{a(x)f(t)}=1\qquad\mbox{uniformly in}\,\Omega.$$
\end{enumerate}

\begin{defn}
	A function $f\in C^1([0,+\infty))$ satisfies \textit{assumption (A)} if
	\begin{enumerate}[\text{A}1)]
		\item $f>0$ and $f'\big|_{[M,+\infty)}\geq0$ for some $M>0$;
		\item $f$ is superlinear at $\infty$, that is, $\displaystyle\lim_{t\to+\infty}{\tfrac{f(t)}{t}}=+\infty$;
		\item there exists $\displaystyle\lim_{t\to+\infty}{\tfrac{f'(t)}{f(t)}}\in[0,+\infty)$.
	\end{enumerate}
\end{defn}

The assumption (\textit{H2}) may be seen as a generalization of \eqref{GSgrowth}, as we are prescribing a sort of separation of variables at $\infty$, and one retrieves that form provided the potential $a$ stays bounded from 0 near the boundary. However, it is worth to stress here that the potential $a$ might also vanish at $\dOmega$ and, more important, that our analysis is not restricted to a precise profile at $\infty$ as in \cite{Soranzo,RW}, but we include \textit{almost any} growth in $t$ which is, in some sense, controlled from above by $e^{\gamma t}$ for some $\gamma>0$. Indeed, it is possible to show that assumption (\textit{A3}) is equivalent to require a control from above by a suitable exponential function (see Lemma \ref{fexpbeta_lem} below).

\begin{remark}
	\noindent From (\textit{H1})-(\textit{H2}), it follows that for any $\varepsilon>0$ there exists a constant $d_\varepsilon$ such that
	\begin{equation}\label{asymptotic_GROWTH}
	(1-\varepsilon)a(x)f(t)-d_\varepsilon\leq h(x,t)\leq (1+\varepsilon)a(x)f(t)+d_\varepsilon\quad\;\mbox{for all \,$t\geq0$, $x\in\Omega$}.
	\end{equation}
\end{remark}
Although our main result applies for all nonlinearities, it will be useful in the sequel to distinguish among the admissible growths, in dependence of how far they are from being exponential, according to the following definition.

\begin{defn}
	Let $h$ satisfy assumption (\textit{H2}) and $f$ be as therein. We say that $h$ is \textit{subcritical} if
	\begin{equation}\label{subcritical}
		\lim_{t\to+\infty}{\frac{f'(t)}{f(t)}}=0.
	\end{equation}
	On the other hand, we say that $h$ is \textit{critical} if
	\begin{equation}\label{quasicritical}
	\lim_{t\to+\infty}{\frac{f'(t)}{f(t)}}\in(0,+\infty).
	\end{equation}
\end{defn}
\noindent Roughly speaking, the behaviour of the nonlinearities of \textit{subcritical} type largely differs from the exponential map. An example is $h(x,t)=a(x)\log^{\theta}(t+1)t^p e^{t^\alpha}$ with $\theta\geq 0$, $p>1$, $\alpha\in[0,1)$ and $a$ as in (\textit{H2}). The class of \textit{critical} functions gathers instead maps which are quite close to $e^t$ with respect to the second variable, not affecting too much its exponential behaviour; model nonlinearities for this case are $h(x,t)=a(x)\frac{e^{\gamma t}}{(t+1)^q}$ with $q\in\R$, $\gamma>0$. These behaviours are clarified by the following lemma.
\begin{lem}\label{fexpbeta_lem}
	Let $f$ be a nonlinearity satisfying assumption (\textit{A}) and denote $\beta:=\displaystyle\lim_{t\to+\infty}{\tfrac{f'(t)}{f(t)}}$. Then for any $\varepsilon>0$ there are constants $C_\varepsilon,D_\varepsilon>0$ such that
	\begin{equation}\label{fexpbeta}
	D_\varepsilon e^{(\beta-\varepsilon)t}-C_\varepsilon\leq f(t)\leq D_\varepsilon e^{(\beta+\varepsilon)t}+C_\varepsilon.
	\end{equation}
\end{lem}
\begin{proof}
	By definition of $\beta$, for any $\varepsilon>0$ there exists $T_\varepsilon>0$ such that
	$$(\beta-\varepsilon)f(t)\leq f'(t)\leq (\beta+\varepsilon)f(t)$$
	for any $t>T_\varepsilon$. Integrating then on $[T_\varepsilon,t]$, we get
	$$(\beta-\varepsilon)\int_{T_\varepsilon}^tf(s)ds+f(T_\varepsilon)\leq f(t)\leq (\beta+\varepsilon)\int_{T_\varepsilon}^tf(s)ds+f(T_\varepsilon).$$
	By Gronwall lemma from the second inequality we infer
	$f(t)\leq D_\varepsilon^+ e^{(\beta+\varepsilon)t}$ where $D_\varepsilon^+:=f(T_\varepsilon)e^{-(\beta+\varepsilon)T_\varepsilon}$ and, since $f$ is positive and bounded in $[0,T_\varepsilon]$, then we may set $C_\varepsilon^+:=\max_{[0,T_\varepsilon]} f$.
	The first inequality in \eqref{fexpbeta} is proved similarly, applying Gronwall lemma to $-f$, showing $f(t)\geq D_\varepsilon^- e^{(\beta-\varepsilon)t}-C_\varepsilon^-$. Then it is enough to define $D_\varepsilon:=\max\{D_\varepsilon^+,D_\varepsilon^-\}$ and $C_\varepsilon:=\max\{C_\varepsilon^+,C_\varepsilon^-\}$.
\end{proof}
\vskip0.2truecm

\noindent We mainly focus on the biharmonic problem 
\begin{equation}\label{DIRf}\tag{P}
	\begin{cases}
		\Delta^2u=h(x,u)\quad&\mbox{in }\Omega\subset\R^4,\\%
		u=u_n=0\quad&\mbox{on }\dOmega,
	\end{cases}
\end{equation}
\noindent that is \eqref{DIRf_generalpoly} with $m=2$, addressing in Section \ref{SectionPOLY} the discussion of the general problem \eqref{DIRf_generalpoly}, the argument being similar.

\begin{defn}
	We say $u$ is a \textit{weak solution} of \eqref{DIRf} if $u\in\Hoo$ and, for every $\varphi\in\Hoo$:
	\begin{equation}\label{def}
	\int_\Omega\Delta u\Delta\varphi=\int_\Omega h(x,u)\varphi.
	\end{equation}
\end{defn}

\noindent Notice that that \eqref{asymptotic_GROWTH}, \eqref{fexpbeta} and the Adams inequality \cite{Adams} imply that $h(\cdot,u(\cdot))\in L^p(\Omega)$ for any $u\in H^2_0(\Omega)$, $p\ge 1$.
\vskip0.2truecm
The following assumption will be needed to control the behaviour of solutions near the boundary. 
\begin{enumerate}[H3)]
	\item there exist $\bar r,\bar\delta>0$ such that
	\begin{enumerate}[H3a)]
		\item $h(x,\cdot)$ is increasing for $t\geq 0$ for all $x\in\omega_{\bar r}:=\{x\in\ \B\,|\,d(x,\partial\B)<\bar r\}$;
		\item $h(\cdot,t)\in C^1(\omega_r)$ for all $t\geq 0$ and $\nabla_xh(x,t)\cdot\theta\leq 0$ for all $x\in\omega_{\bar r}$, $t\geq0$ and unit vectors $\theta$ such that $|\theta-n(x)|\leq\bar\delta$.
	\end{enumerate}
\end{enumerate}

\noindent Our main results may be then summarized as follows:
\begin{thm}\label{uniformboundblowupBALL}
	Let $\B$ be a ball in $\R^4$ and $h$ satisfy (\textit{H1})-(\textit{H2})-(\textit{H3}). Suppose moreover that one of the following holds:
	\begin{enumerate}[i)]
		\item $a(\cdot)\geq a_0>0$;
		\item $\displaystyle\lim_{t\to+\infty}\tfrac{f(t)}{t^\alpha}=+\infty$, for some $\alpha>1$.
	\end{enumerate}
	Then there exists $C>0$ such that $\|u\|_{L^\infty(\B)}\leq C$ for all weak solutions of \eqref{DIRf}.
\end{thm}
\indent Let us now define
\begin{equation}\label{F}
F(t):=\int_{0}^{t}f(s)ds\qquad\mbox{and}\qquad H(x,t):=\int_{0}^{t}h(x,s)ds.
\end{equation}

\begin{thm}\label{uniformboundblowupOMEGA}
	Let $\Omega\subset\R^4$ be a bounded $C^{4,\gamma}$ smooth domain, $\gamma\in(0,1)$ and $h$ satisfy (\textit{H1})-(\textit{H2}) with $0<a_0\leq a(\cdot)\in C(\Omegabar)$. Assume that one of the following conditions holds:
	\begin{enumerate}[\indent 1)]
		\item $h$ is subcritical according to \eqref{subcritical};
		\item $h$ is critical according to \eqref{quasicritical} and there exists a neighborhood $\omega$ of $\dOmega$ and functions $0\leq B\in L^\infty(\omega)$, $0\leq D\in L^1(\omega)$, such that
		\begin{equation}\tag{H4}\label{hpnablaH}
			|\nabla_x H(x,t)|\leq B(x)F(t)+D(x),\quad\mbox{for any}\quad t\geq0,\,x\in\omega.
		\end{equation}
	\end{enumerate}
	Then, for any $\Lambda>0$ there exists $C>0$ depending on $\Lambda$ such that  $\|u\|_{L^\infty(\Omega)}\leq C$ for any solution of \eqref{DIRf} such that $\int_\Omega h(x,u)dx\leq\Lambda$.
\end{thm}
\noindent Before going into the details of the proofs, let us make some remarks about these two results.
\begin{enumerate}
	\item Theorem \ref{uniformboundblowupOMEGA} extends the a-priori bound of Theorem \ref{uniformboundblowupBALL} to the case of a smooth bounded domain, provided a control on the energy of solutions is assumed. This is a fairly natural assumption in the framework of the conformal geometry: see for instance \cite{ARS,BM,JY,Mart_cc,MP,RobWei}.
	\item The conditions (\textit{H3a})-(\textit{H3b}) provide a control near the boundary of $h$ and in broad terms they state that $h$ should be decreasing in $x$ along outer directions and increasing in $t$.
	\item The additional assumption \eqref{hpnablaH}, when applied to a model nonlinearity  of the form $h(x,t)=a(x)f(t)$, is nothing but a uniform control on $\nabla a$, and is thus equivalent to require $a\in W^{1,\infty}$ in a neighborhood of the boundary.
	\item By standard elliptic theory and assumption (\textit{H1}), the uniform bounds hold also in $C^{3,\alpha}$ for any $\alpha\in(0,1)$.
	\item As mentioned before, the assumptions on $f$ gather all superlinear profiles which can be controlled by a map $t\to e^{\gamma t}$ for some $\gamma>0$. This bound from above on the growth of $f$ reveals to be \textit{sharp} in the class of \textit{distributional} solutions. Indeed, in the spirit of Brezis and Merle, in Section \ref{Sezcounterex} we consider a problem of the kind \eqref{DIRf} with a nonlinearity having faster growth at $\infty$ and we show a distributional solution which is not even bounded on $\Omega$. However, the behaviour of sequences of \textit{weak} solutions for such nonlinearities is still an open question.

\end{enumerate}

\vskip0.2truecm
As briefly recalled at the beginning, the main motivation to study a-priori bounds for semilinear elliptic problems is to infer the existence of solutions. The subsequent Theorem \ref{Existence} is in fact established combining the Krasnosel'skii topological degree theory with Theorems \ref{uniformboundblowupBALL} and \ref{uniformboundblowupOMEGA}. We remark that the existence of solutions for semilinear Dirichlet problems like \eqref{DIRf_generalpoly} can be obtained up to the critical growth $t\mapsto e^{t^2}$ by variational methods, see \cite{Lakkis,LamLu}.\\
\indent To fix the notation, here and in the sequel, we denote by $\tilde\lambda_1(\Omega)$ (resp. $\lambda_1(\Omega)$) the first eigenvalue of $\Delta^2$ (resp. $-\Delta$) in the domain $\Omega$ subjected to Dirichlet boundary conditions, omitting the domain whenever it is clear from the context.
\begin{thm}\label{Existence}
	Let $\B$ be a ball in $\R^4$ and $h$ satisfy the assumptions of Theorem \ref{uniformboundblowupBALL}. Suppose moreover that $h(\cdot, t)\in C(\bar\B)$ (resp. $C^{0,\gamma}(\bar\B)$ for some $\gamma\in(0,1)$) for any $t\geq0$ and
	\begin{equation}\label{HP}
	\limsup_{t\to0^+}\frac{h(x,t)}{t}<\tilde\lambda_1\quad\mbox{uniformly in $x\in\B$}.
	\end{equation}
	Then, problem \eqref{DIRf} admits a positive strong (resp. classical) solution.
\end{thm}
\begin{remark}
	The assumptions of Theorem \ref{Existence} are, for instance, satisfied by $h(x,t)=a(x)t^pe^{t^\alpha}$ for any $\alpha\in[0,1)$, $p>1$ regarding the subcritical context, or $h(x,t)=a(x)t^p \frac{e^{\theta t}}{(t+1)^\gamma}$ for any $\gamma\geq0$ and $p,\theta>0$ for the critical framework.
\end{remark}
\vskip0.2truecm
This article is organized as follows. In Section \ref{SectionNaPR} we establish the uniform boundedness of solutions near the boundary and the $L^1$ control for the energy of solutions. In Section \ref{SectionBlowup} the blow-up technique is outlined both for subcritical and critical nonlinearities, leading to the proof of Theorem \ref{uniformboundblowupBALL}. In Sections \ref{SectionDomains} and \ref{SectionPOLY} we consider some extensions and in particular we prove Theorem \ref{uniformboundblowupOMEGA}. Theorem \ref{Existence} is established in Section \ref{SectionEXISTENCE}, while a counterexample concerning the behaviour of solutions above the critical case is shown in Section \ref{Sezcounterex}. Finally, we collect in Section \ref{SectionOP} some problems which are left open.

\vskip0.2truecm
\underline{\textbf{Notation:}} Let $f,g:\Omega\to\R^+$: we say that $f\preceq g$ if there exists a constant $c>0$ such that $f(t)\leq cg(t)$ for all $t\in\Omega$, and we write $f\simeq g$ if both $f\preceq g$ and $g\preceq f$ hold. We define $f\wedge g:=\min\{f,g\}$. Moreover, $d_\Omega(x)$ denotes the distance from $x\in\Omega$ to the boundary $\dOmega$. Finally, $c,C$ indicate generic constants, whose value may vary from line to line, and also within the same line.
\vskip0.4truecm
\underline{\textbf{Acknowledgements:}} The paper was prepared when the first author was employed at Universit\`{a} degli Studi di Padova and the second author at Aix-Marseille Universit\'{e}. G.Mancini was supported by the Swiss National Science Foundation, projects nr. PP00P2-144669 and PP00P2-170588/1. G.Romani has received funding from Excellence Initiative of Aix-Marseille Universit\'{e} - A*MIDEX, a French "Investissements d'Avenir" programme.

\section{Uniform estimates near the boundary and energy estimates on a ball} \label{SectionNaPR}

The main aim of this section is to prove the following result.
\begin{prop}\label{NEWSection1}
	Let $\mathcal B$ the unit ball in $\R^4$. Let $h$ satisfy the assumptions of Theorem \ref{uniformboundblowupBALL}. Then there exist $C_1,\Lambda>0$ and  a small neighborhood $\omega\subset\B$ of $\dB$ such that
	\begin{equation*}
	\|u\|_{L^\infty(\omega)}\leq C_1\qquad\mbox{and}\qquad\int_\B h(x,u)dx\leq\Lambda
	\end{equation*}
	for all (positive) weak solutions $u$ of \eqref{DIRf}.
\end{prop}

\noindent This result prevents boundary blow-up and yields an energy control on solutions. This will be essential for the blow-up technique in Section \ref{SectionBlowup}. The proof is inspired by some arguments in \cite{dFLN} for second-order elliptic problems.
Since for higher-order problems the maximum principle does not hold in general, we have to restrict to the class of positivity preserving domains. However, in oder to apply a moving-planes technique, we shall also need pointwise estimates of the Green function of $\Delta^2$ with Dirichlet boundary conditions, which are known only when $\Omega$ is a ball.

\vskip0.2truecm
In the following, we will often make use of the Green function $G_{\Omega}$ of the biharmonic operator with Dirichlet boundary conditions. We recall that $G_{\Omega}:\Omegabar\times\Omegabar\setminus \{(x,x) :
x\in \overline{\Omega}\}\rightarrow\R$ is defined as the unique function such that for any $g\in L^2(\Omega)$,
\begin{equation*}
u(x):=\int_\Omega G_{\Omega}(x,y)\,g(y)\,dy
\end{equation*}
is the unique solution in $\Hoo$ of the equation $\Delta^2u=g$ in $\Omega$.\\
We collect here some results about pointwise estimates of $G_{\Omega}$ and of its derivatives. They go back to the work of Krasovski{\u\i} \cite{Kraso} later on refined by Dall'Acqua and Sweers in \cite{DAS}. However, very smooth boundaries are required therein. This formulation which assumes less regularity can be deduced combining \cite[Theorem 4]{GR} and \cite[Theorem 2]{GRholes}.
\begin{lem}\label{Greenspropetry}
	Let $\Omega\subset\R^4$ be a bounded domain of class $C^{4,\gamma}$ for some $\gamma\in(0,1)$. There exists a positive constant $C$ depending on the domain, such that for all $x,y\in\Omega$, $x\neq y$, there holds
	\begin{equation}\label{Greenspropetry_0}
	|G_\Omega(x,y)|\leq C\log\bigg(2+\dfrac{1}{|x-y|}\bigg),
	\end{equation}
	$$|\nabla^iG_\Omega(x,y)|\leq \frac{C}{|x-y|^i},\quad\mbox{for any}\,\,i\geq1.$$
\end{lem}
\noindent In the special case of $\Omega=\B$, thanks to the explicit Boggio's formula, also the estimate from below of $G_\B$ may be obtained. Therefore, the following sharp two-sided estimate holds:

\begin{lem}[\cite{GGS}, Theorem 4.6]\label{2sidedestimateball}
	In $\overline\B\times\overline\B$ we have
	\begin{equation}\label{2sided}
	G_\B(x,y)\simeq \log\bigg(1+\dfrac{d_\B(x)^2d_\B(y)^2}{|x-y|^4}\bigg).
	\end{equation}
\end{lem}

\vskip0.2truecm
\indent In the whole section, we suppose $\B$ and $h$ as in Proposition \ref{NEWSection1} and we denote by $\tilde\varphi_1$ the first eigenfunction of $\Delta^2$ in $\B$ subjected to Dirichlet boundary conditions. Notice that $\tilde\varphi_1>0$ (see \cite[Theorem 3.7]{GGS}).

\vskip0.2truecm
\indent We split the proof of Proposition \ref{NEWSection1} in several steps. First, we obtain a local $L^1$ estimate for the right-hand side of equation \eqref{DIRf}. Here the assumptions (\textit{i})-(\textit{ii}) of Theorem \ref{uniformboundblowupBALL} play a role, together with pointwise Green function estimates from below.
\begin{lem}[Local energy estimate]\label{lemmaDir1-B}
	There exists a constant $C>0$ such that
	\begin{equation}\label{localL1bdd}
	\int_\B h(x,u)\tilde\varphi_1= \tilde\lambda_1\int_\B u\tilde\varphi_1\leq C,
	\end{equation}
	for all weak solutions $u$ of \eqref{DIRf}.
\end{lem}
\begin{proof}
	The equality in \eqref{localL1bdd} follows simply by testing \eqref{def} with $\varphi=\tilde\varphi_1$:
	$$\int_\B h(x,u)\tilde\varphi_1=\int_\B\Delta u\Delta\varphitilde=\lambdatilde\int_\B u\varphitilde.$$
	\indent \textit{Suppose (i) holds.}	By the superlinearity of $f$, for any $M>0$, there exists $t_0(M)$ such that $f(t)>Mt$ for any $t>t_0$. Therefore by \eqref{asymptotic_GROWTH},
	\begin{align*}
	\int_\B u\tilde\varphi_1&=\int_{\{u\leq t_0\}}u\tilde\varphi_1+\int_{\{u>t_0\}}u\tilde\varphi_1\leq t_0\|\tilde\varphi_1\|_1+\frac1M\int_\B f(u)\tilde\varphi_1\\
	&\leq C_1(t_0)+\frac 1{Ma_0}\int_\B {a(x)}f(u)\tilde\varphi_1\leq C_2(t_0,M)+\frac 2{a_0{M}}\int_\B h(x,u)\tilde\varphi_1.
	\end{align*}
	Choosing now $M=\tfrac{4\tilde\lambda_1}{a_0}$, \eqref{localL1bdd} follows.\\
	\indent\textit{Suppose (ii) holds.} We follow here some ideas of \cite{DST}, see also \cite{Soup}. Let $c_1,c_2$ be positive constants such that
	$$c_1d_\B^2(x)\leq\tilde\varphi_1(x)\leq c_2d_\B^2(x)$$
	for all $x\in\B$, see \cite[Lemma 3]{CS}. By \eqref{2sided} we get the pointwise estimate of the Green function (see \cite[Remark 3]{GS_nav2})
	$$G_\B(x,y)\geq C\log\bigg(2+\frac{d_\B(y)}{|x-y|}\bigg)\bigg(1\wedge\frac{d_\B^2(x)d_\B^2(y)}{|x-y|^{4}}\bigg)\geq C d_\B^2(x)^2 d_\B^2(y).$$
 Using the representation formula, for all solutions $u$ of \eqref{DIRf} there holds
	\begin{equation}\label{u_rappr}
	\begin{split}
	u(x)&
	\geq Cd_\B^2(x)\int_{\B} d_\B^2(y)h(y,u(y))dy\geq Cc_2^{-2}\tilde\varphi_1(x)\int_{\B} \tilde\varphi_1(y)h(y,u)dy.
	\end{split}
	\end{equation}
	Moreover, by \eqref{asymptotic_GROWTH} with $\varepsilon=\tfrac12$ and (\textit{ii}), we have
	\begin{equation*}
	\begin{split}
	\int_{\B} h(x,u)\tilde\varphi_1(x)dx&\geq\frac12\int_{\B} a(x)f(u)\tilde\varphi_1(x)dx-d\int_{\B}\tilde\varphi_1(x)dx\\
	&\geq\int_{\B} a(x)(Cu^\gamma(x)-D)\tilde\varphi_1(x)dx-d\int_{\B}\tilde\varphi_1(x)dx,
	\end{split}
	\end{equation*}
	where $C$ and $D$ are suitable positive constant.
	Therefore, by \eqref{u_rappr},
	\begin{equation*}
	\begin{split}
	&\int_{\B} h(x,u)\tilde\varphi_1(x)dx+C(\|a\|_1,\gamma)\geq C\int_{\B} a(x)u^\gamma(x)\tilde\varphi_1(x)dx\\
	&\geq C\bigg(\int_{\B} a(x)\tilde\varphi_1^{1+\gamma}(x)dx\bigg)\bigg(\int_{\B} h(x,u)\tilde\varphi_1(x)dx\bigg)^\gamma.
	\end{split}
	\end{equation*}
	The inequality \eqref{localL1bdd} follows, noticing that $\gamma>1$ and all constants are positive.
\end{proof}

The aim now is to prove that there exists a uniform neighborhood of the boundary where solutions are decreasing along suitable outer directions. The main tool is a moving-planes technique inspired by \cite{BGW}, which is strongly based on pointwise Green functions estimates of the ball. Indeed, these ensure to get rid of problems connected to the lack of a maximum principle for domains of small measure.\\
\indent Let us introduce some notation: for each $\lambda\in(0,1)$ we set
\begin{equation*}
\Sigma_\lambda:=\{x\in\B\,|\,x_1>1-\lambda\},\quad T_\lambda:=\{x\in\B\,|\,x_1=1-\lambda\},\quad
\Sigma_{\lambda'}:=\{x\in\B\,|\,\bar x\in\Sigma_\lambda\},
\end{equation*}
where $\bar x$ is defined as the reflection of $x$ through $T_\lambda$.
By convexity, it is easy to see that $\Sigma_\lambda\cup\Sigma_{\lambda'}\subset\omegar$ provided $\lambda$ is small enough, where $\bar r$ is defined in (\textit{H3}).

\begin{lem}\label{MovingPlanes}
	There exists $r\in(0,\bar r)$ and $\delta\in (0,\bar\delta)$ such that in $\omega_r:=\{x\in\B\,|\,d(x,\dB)<r\}$ there holds $\nabla u(x)\cdot\theta\leq 0$ for all $|\theta-n(x)|<\delta$ for any solution of \eqref{DIRf}.
\end{lem}
\begin{proof}
	 Using a compactness argument, the statement may be reduced to prove that there exists a neighborhood $\mathcal U$ of $X_1:=(1,0,0,0)$ such that $\frac{\partial u}{\partial x_1}(x)<0$ for any $x\in\mathcal U$ and any $u$ solution of \eqref{DIRf}. Define the set
	 $$\Gamma_u:=\left\{\lambda\in(0,1)\,|\,u(x)<u(\bar x)\; \forall x\in\Sigma_\lambda,\;\frac{\partial u}{\partial x_1}(x)<0\; \forall x\in T_\lambda\right\}.$$
	 Our aim is to prove that for any solution $u$ there holds $\Gamma_u\supseteq(0,\tfrac{\bar r}2)$. Setting $S_\lambda:=\B\setminus(\Sigma_\lambda\cup T_\lambda)$, and extending $u$ by 0 outside $\B$, then we can rewrite $\Gamma_u$ as
	 $$\Gamma_u:=\left\{\lambda\in(0,1)\,|\,u(x)>u(\bar x)\; \forall x\in S_\lambda,\;\frac{\partial u}{\partial x_1}(x)<0\; \forall x\in T_\lambda\right\}.$$
	 The asymptotic behaviour of the Green function in Lemma \ref{2sidedestimateball} implies that for any $x_0\in\dB$ and for any outer direction $\theta$ we have $\left(\frac{\partial}{\partial\theta}\right)^2u(x_0)>0$ (see \cite[Theorem 3.2]{GS_pp}). In particular, choosing $x_0=X_1$ and $\theta=x_1$, one finds $\frac{\partial^2 u}{\partial x_1^2}(x)>0$ in a small neighborhood of $X_1$. Thus, by zero Dirichlet boundary conditions, one infers $\frac{\partial u}{\partial x_1}(x)<0$ up to a smaller neighborhood. This means that for any $u$, the set $\Gamma_u$ is nonempty. Let us suppose by contradiction that
	 $$\lambdabar:=\sup\big\{\lambda\in (0,1)\, | \,(0,\lambda)\subseteq \Gamma_u \big\}<\frac{\bar r}2.$$
	 Then, for any $x\in S_{\lambdabar}$ we have $u(x)\geq u(\bar x)$. We firstly claim that $\frac{\partial u}{\partial x_1}(x)<0$ on $T_{\lambdabar}$. Indeed, after extending $h$ and $G_\B$ to 0 outside their domains,
	 \begin{equation}\label{dudx1}
	 \begin{split}
	 \frac{\partial u}{\partial x_1}(x)&=\int_{\Sigma_	{\lambdabar}}\frac{\partial G_\B}{\partial x_1}(x,y)h(y,u(y))dy+\int_{S_{\lambdabar}}\frac{\partial G_\B}{\partial x_1}(x,y)h(y,u(y))dy\\
	 &=\int_{S_{\lambdabar}}\left(\frac{\partial G_\B}{\partial x_1}(x,y)h(y,u(y))+\frac{\partial G_\B}{\partial x_1}(x,\bar y)h(\bar y,u(\bar y))\right)dy.
	 \end{split}
	 \end{equation}
	 As in $S_{\lambdabar}$ there holds $u(y)\geq u(\bar y)$ and noticing that $\Sigma_{\lambdabar}\cup\Sigma_{\lambdabar}'\subset\omegar$ as $\lambdabar<\frac{\bar r}2$, then by assumption (\textit{H3}) we deduce
	 \begin{equation*}
	 h(y,u(y))\stackrel{(H3a)}{\geq} h(y,u(\bar y))\stackrel{(H3b)}{\geq} h(\bar y,u(\bar y))\qquad\mbox{for all }y\in S_{\lambdabar}.
	 \end{equation*}
	 More precisely, we may also infer that one of the two inequalities
	 \begin{equation}\label{h_hbar}
	 h(y,u(y))\geq h(\bar y,u(\bar y))\geq 0\qquad\mbox{for all }y\in S_{\lambdabar}
	 \end{equation}
	 is strict for a set $\mathcal O_{\lambdabar}$ of positive measure. Indeed, if not, we would have $h(y,u(y))=0$ a.e. on $S_{\lambdabar}$ and in turns by \eqref{h_hbar} $h(y,u(y))=0$ a.e. on $\B$. Thus, $u$ would be a positive solution for $\Delta^2u=0$ in $\B$ with zero Dirichlet boundary conditions, a contradiction.\\
	 \indent Moreover, we need the following pointwise properties of the Green function $G_\B$.
	 \begin{lem}[\cite{BGW}, Lemma 3]\label{BGW_lemma}
	 	Let $\lambda\in(0,1)$. For all $x\in T_{\lambda}$ and $y\in S_\lambda$ there hold
	 	\begin{equation}\label{BGW_1}
	 	\frac{\partial G_\B}{\partial x_1}(x,y)<0;
	 	\end{equation}
	 	\begin{equation}\label{BGW_2}
	 	\frac{\partial G_\B}{\partial x_1}(x,y)+\frac{\partial G_\B}{\partial x_1}(x,\bar y)\leq 0.
	 	\end{equation}
	 \end{lem}
	 \noindent Hence, if the first inequality in \eqref{h_hbar} is strict in $\mathcal O_{\lambdabar}$, from \eqref{dudx1} we have
	 \begin{equation*}
	 \frac{\partial u}{\partial x_1}(x)\stackrel{\eqref{BGW_1}}{<}\int_{S_{\lambdabar}}\left(\frac{\partial G_\B}{\partial x_1}(x,y)+\frac{\partial G_\B}{\partial x_1}(x,\bar y)\right)h(\bar y,u(\bar y))dy\stackrel{\eqref{BGW_2}}{\leq}0;
	 \end{equation*}
	 otherwise, if the second inequality in \eqref{h_hbar} is strict in $\mathcal O_{\lambdabar}$,
	 \begin{equation*}
	 \frac{\partial u}{\partial x_1}(x)\stackrel{\eqref{BGW_1}}{\leq}\int_{S_{\lambdabar}}\left(\frac{\partial G_\B}{\partial x_1}(x,y)+\frac{\partial G_\B}{\partial x_1}(x,\bar y)\right)h(\bar y,u(\bar y))dy\stackrel{\eqref{BGW_2}}{<}0;
	 \end{equation*}
	 As in both cases the sign is strict, our claim is proved. Moreover, with the same compactness argument as in \cite[Lemma 8]{BGW}, one can actually slide a little bit inwards the hyperplane and obtain the same sign for the derivative of $u$: one shows indeed that there exists $\gamma\in(0,1-\lambdabar)$ such that
	 \begin{equation}\label{dudx1_big}
	 \frac{\partial u}{\partial x_1}(x)<0\;\; \mbox{on}\; T_s \quad\mbox{for any}\; s\in (\lambdabar,\lambdabar+\gamma).
	 \end{equation}
	 With such information, we want to prove a contradiction with the maximality of $\lambdabar$. Again, we need some pointwise properties of $G_\B$.
	 \begin{lem}[\cite{FGW}, Lemma 3]\label{FGW_lemma}
	 	Let $\lambda\in(0,1)$. For all $x,y\in S_\lambda$, $x\not=y$, there hold
	 	\begin{equation}\label{FGW_1}
	 	G_\B(x,y)>\max\{G_\B(x,\bar y),G_\B(\bar x,y)\};
	 	\end{equation}
	 	\begin{equation}\label{FGW_2}
	 	G_\B(x,y)-G_\B(\bar x,\bar y)>|G_\B(x,\bar y)-G_\B(\bar x,y)|.
	 	\end{equation}
	 \end{lem}
	 \noindent We claim that $u(x)>u(\bar x)$ for all $x\in S_{\lambdabar}$. Indeed,
	 \begin{equation*}
	 \begin{split}
	 u(x)-u(\bar x)&=\int_\B \left[G_\B(x,y)-G_\B(\bar x,y)\right]h(y,u(y))dy\\
	 &=\int_{S_{\lambdabar}}\big(\left[G_\B(x,y)-G_\B(\bar x,y)\right]h(y,u(y))+\left[G_\B(x,\bar y)-G_\B(\bar x,\bar y)\right]h(\bar y,u(\bar y))\big)dy.
	 \end{split}
	 \end{equation*}
	 Then, if the first inequality in \eqref{h_hbar} is strict in $\mathcal O_{\lambdabar}$, then we get
	 \begin{equation}\label{ux_ubarx}
	 u(x)-u(\bar x)\stackrel{\eqref{FGW_1}}{>}\int_{S_{\lambdabar}}\left[G_\B(x,y)-G_\B(\bar x,y)+G_\B(x,\bar y)-G_\B(\bar x,\bar y)\right]h(\bar y,u(\bar y))dy\stackrel{\eqref{FGW_2}}{\geq} 0.
	 \end{equation}
	 Otherwise, if it is the second inequality in \eqref{h_hbar} to be strict in $\mathcal O_{\lambdabar}$, we obtain the same result exchanging the strict signs in \eqref{ux_ubarx}. This, combined with \eqref{dudx1_big}, {by means of a standard compactness argument}, shows that $\Gamma_u\supseteq(0,\lambdabar+\varepsilon)$ for a suitable $\varepsilon\in(0,\gamma)$, contradicting the maximality of $\lambdabar$.
\end{proof}

\vskip0.2truecm
The behaviour of solutions near the boundary found in Lemma \ref{MovingPlanes} and the local $L^1$ bound of Lemma \ref{lemmaDir1-B} allow to prove Proposition \ref{NEWSection1} in the following two lemmas.

\begin{lem}\label{boundaryestDIR}
	There exists a neighborhood $\omega$ of $\dB$ and $C_1>0$ such that $\|u\|_{L^\infty(\omega)}\leq C_1$ for all weak solutions $u$ of \eqref{DIRf}.
\end{lem}
\begin{proof}
	By Lemma \ref{MovingPlanes}, arguing as in \cite{dFLN}, one may infer that for every $x\in\omega_r$ there exists a set $I_x$ and a constant $\gamma>0$ such that $|I_x|\geq\gamma$, $I_x\subseteq \B\setminus\omega_\frac{r}{2}$ and $u(y)\geq u(x)$ for all $y\in I_x$.
	Taking $x\in\omega_r$, by positivity of $\tilde\varphi_1$ and Lemma \ref{lemmaDir1-B},
	$$C\geq\int_\B h(x,u)\tilde\varphi_1 dx\geq\tilde\lambda_1\int_{\B\setminus\omega_\frac{r}{2}} u\tilde\varphi_1\geq\tilde\lambda_1\min_{\B\setminus\omega_\frac{r}{2}}\tilde\varphi_1\int_{I_x}u(y)dy\geq c(\B)\gamma u(x),$$
	which implies the uniform $L^\infty$ boundedness of $u$ in $\omega_r$.
\end{proof}

\begin{lem}\label{uniformLambda}
	There exists a constant $\Lambda>0$ such that $\int_\B h(x,u)dx\leq\Lambda$ for all weak solutions $u$ of \eqref{DIRf}.
\end{lem}
\begin{proof}
	Let $\omega$ and $C_1$ be as in Lemma \ref{boundaryestDIR}. By (\textit H1), Lemma \ref{boundaryestDIR}, the positivity of $\tilde\varphi_1$ and Lemma \ref{lemmaDir1-B}, one has
	\begin{equation*}
	\begin{split}
	\int_\B h(x,u)dx&=\int_{\omega} h(x,u)dx+\int_{\B\setminus\omega} h(x,u)dx\\
	&\leq  \|h\|_{L^{\infty}(\Omega\times[0,C_1])}+\dfrac{1}{m(\omega)}\int_{\B\setminus\omega} h(x,u)\tilde\varphi_1dx\\
	&\leq \|h\|_{L^{\infty}(\Omega\times[0,C_1])}+\dfrac{1}{m(\omega)}\int_\B h(x,u)\tilde\varphi_1dx\leq \Lambda(\B,h)
	\end{split}
	\end{equation*}
	having defined $m(\omega):=\min_{\B\setminus\omega}\tilde\varphi_1>0$.
\end{proof}

\section{Uniform bounds inside the domain}\label{SectionBlowup}
The main goal of this section is to show how a scaling argument ensures that solutions of \eqref{DIRf} are uniformly bounded away from the boundary. This, together with the uniform boundedness near $\dOmega$ obtained in Proposition \ref{NEWSection1} will prove Theorem \ref{uniformboundblowupBALL}.

\indent Although the analysis that we present here concerns the problem \eqref{DIRf} in $\B$, in this section we use the notation $\Omega$ to indicate the ball. This choice is motivated by Section \ref{SectionDomains}, where we show that the same argument can be applied also for general bounded smooth domains to prove Theorem \ref{uniformboundblowupOMEGA}.
\vskip0.2truecm
Let us suppose by contradiction that there exists a sequence $(u_k)_{k\in\N}$ of solutions of problem \eqref{DIRf} and a sequence of maximum points $(x_k)_{k\in\N}\subset\Omega$ such that
\begin{equation}\label{maxima}
u_k(x_k)=\|u_k\|_{L^\infty(\Omega)}=:M_k\nearrow+\infty.
\end{equation}
Since $\Omega$ is bounded, up to a subsequence $x_k\to x_\infty$ with $d(x_\infty,\dOmega)>0$, by Proposition \ref{NEWSection1}. Moreover, we define the 
rescaled functions $v_k:\Omega_k\to\R$ as
\begin{equation}\label{vk}
v_k(x):=u_k(x_k+\mu_kx)-M_k,
\end{equation}
where the scaling is
\begin{equation}\label{muk}
\mu_k:=\dfrac1{(f(M_k))^{1/4}}\rightarrow0\quad\mbox{ as }k\to+\infty
\end{equation}
and the expanding domains are $\Omega_k:=\frac{\Omega-x_k}{\mu_k}$. Notice that $x_\infty\in\Omega$ implies $\Omega_k\nearrow\R^4$. Then we have
\begin{equation}\label{Delta2vk1}
\begin{split}
|\Delta^2v_k(x)|&=\mu_k^4|(\Delta^2u_k)(x_k+\mu_kx)|=\dfrac{h(x_k+\mu_kx,u_k(x_k+\mu_kx))}{f(M_k)}\\
&\leq(1+\varepsilon)a(x_k+\mu_kx)\dfrac{f(u_k(x_k+\mu_kx))}{f(M_k)}+\dfrac{d_\varepsilon}{M_k},
\end{split}
\end{equation}
by \eqref{asymptotic_GROWTH}, so it is uniformly bounded.
\begin{lem}\label{grad_limitato}
	Let $x\in B_R(0)$. There holds $|\nabla^i v_k(x)|\leq C(R)$ for any $i\in\{0,1,2,3\}$.
\end{lem}
\begin{proof}
Firstly, let us take $i\in \{1,2,3\}$. By the representation formula for derivatives and Lemma \ref{Greenspropetry},
	\begin{equation*}
	\begin{split}
	|\nabla^i v_k(x)|&=|\mu_k^i\nabla^i u_k(x_k+\mu_kx)|=\mu_k^i\bigg|\int_\Omega\nabla^i_xG_\Omega(x_k+\mu_kx,y)h(y,u_k(y))dy\bigg|\\
	&\leq C\mu_k^i\int_{\Omega\setminus B_{2R\mu_k}(x_k)}\dfrac{h(y,u_k(y))}{|x_k+\mu_kx-y|^i}dy+C\mu_k^i\int_{B_{2R\mu_k}(x_k)}\dfrac{h(y,u_k(y))}{|x_k+\mu_kx-y|^i}dy.
	\end{split}
	\end{equation*}
	In $\Omega\setminus B_{2R\mu_k}(x_k)$ there holds $|x_k+\mu_kx-y|\geq|y-x_k|-\mu_k|x|\geq2R\mu_k-R\mu_k=R\mu_k$, while in $B_{2R\mu_k}(x_k)$ we have $f(u_k(y))\leq f(M_k)=\mu_k^{-4}$ (this follows from (\textit{A1})). Hence, by Proposition \ref{NEWSection1} and \eqref{asymptotic_GROWTH} with $\varepsilon=1$,
	\begin{equation}\label{proof4.2}
	\begin{split}
	|\nabla^i v_k(x)|&\leq CR^{-i}\Lambda + C\mu_k^i\int_{B_{2R\mu_k}(x_k)}\dfrac{2a(y)f(u_k(y))+d}{|x_k+\mu_kx-y|^i}dy\\
	&\leq CR^{-i}\Lambda +C(2\|a\|_\infty+d\mu_k^4)\mu_k^{i-4}\int_{B_{2R\mu_k}(x_k)}\dfrac{1}{|x_k+\mu_kx-y|^i}dy.
	\end{split} 
	\end{equation}
	Using the change of variable $y=\mu_kz+x_k$, the last integral becomes
	$$\int_{B_{2R\mu_k}(x_k)}\dfrac1{|x_k+\mu_kx-y|^i}dy=\int_{B_{2R}(0)}\dfrac{1}{\mu_k^i|x-z|^i}\mu_k^4dz=\mu_k^{4-i}\int_{B_{2R}(0)}\dfrac{1}{|z-x|^i}dz.$$
	Inserting it into \eqref{proof4.2}, we obtain
	\begin{equation*}
	|\nabla^iv_k(x)|\leq CR^{-i}\Lambda+(2\|a\|_\infty+d\mu_k^4)C\int_0^{2R}\rho^{3-i}d\rho,
	\end{equation*}
	which is finite for $i\in\{1,2,3\}$ since $\mu_k\to0$ as $k\to+\infty$. Taking finally $i=0$ and $x\in B_R(0)$, we have
	$$|v_k(x)|=|v_k(x)-v_k(0)|\leq \sup_{B_R(0)}|\nabla v_k||x|\leq C(R).$$
\end{proof}
As we are able to control both $\Delta^2v_k$ and, locally in $\R^4$, $\nabla^iv_k$, then the local boundedness of the sequence $(v_k)_{k\in\N}$ is achieved by the following Lemma.
\begin{lem}(\cite{RW}, Corollary 6)\label{RWthm}
	Let $\Omega=B_R(0)\subset\R^N$, $m\in\N$, $h\in\Lp(\Omega)$ for some $p\in(1,+\infty)$ and suppose $u\in W^{2m,p}(\Omega)$ satisfies
	$$(-\Delta)^m u=h\,\,\mbox{ in }\,\Omega.$$
	Then there exists a constant $C=C(R,N,p,m)$, such that for any $\delta\in(0,1)$,
	\begin{equation*}
	\|u\|_{W^{2m,p}(B_{\delta R}(0))}\leq\dfrac{C}{(1-\delta)^{2m}}(\|h\|_{\Lp(B_R(0))}+\|u\|_{\Lp(B_R(0))}).
	\end{equation*}
\end{lem}
Applied in our context, using \eqref{Delta2vk1} and Lemma \ref{grad_limitato}, this implies the boundedness of $(v_k)_{k\in\N}$ in $W^{4,p}_{loc}(\R^4)$ so, by compact embedding, there exists $v\in C^3(\R^4)$ such that, {for a subsequence still denoted by $(v_k)_k$,} we have $v_k\to v$ in $C^{3,\gamma}_{loc}(\R^4)$ for any $\gamma\in(0,1)$, satisfying $v\leq0$ and $v(0)=0$. Looking for the equation satisfied by $v$ in $\R^4$, one may rewrite \eqref{Delta2vk1} and obtain
\begin{equation}\label{Delta2vk10}
\begin{split}
\Delta^2v_k(x)&\leq(1+\varepsilon)a(x_k+\mu_kx)e^{\log(f(u_k(x_k+\mu_kx)))-\log(f(M_k))}+\tfrac{d_\varepsilon}{M_k}.
\end{split}
\end{equation}
Taking the first-order Taylor expansion of $\log\circ f$ around $M_k$, one finds
\begin{equation*}
\begin{split}
&\log(f(u_k(x_k+\mu_kx)))=\log(f(M_k))+\dfrac{f'(z_k(x))}{f(z_k(x))}(u_k(x_k+\mu_kx)-M_k),
\end{split}
\end{equation*}
where $z_k(x):=M_k+\theta_k(x)(u_k(x_k+\mu_kx)-M_k)=M_k+\theta_k(x) v_k(x)$, $\theta_k(x)\in(0,1)$. Hence, \eqref{Delta2vk10} becomes 
\begin{equation}\label{Delta2vk100}
\Delta^2v_k(x)\leq(1+\varepsilon)a(x_k+\mu_kx)e^{\frac{f'(z_k(x))}{f(z_k(x))}v_k(x)}+\tfrac{d_\varepsilon}{M_k}.
\end{equation}
Analogously, the following lower bound holds:
\begin{equation}\label{Delta2vk11}
\begin{split}
\Delta^2v_k(x)\geq(1-\varepsilon)a(x_k+\mu_kx)e^{\frac{f'(z_k(x))}{f(z_k(x))}v_k(x)}-\tfrac{d_\varepsilon}{M_k}.
\end{split}
\end{equation}
Since $v_k\to v$ uniformly on compact sets and $M_k\to+\infty$, then $z_k(x)\to+\infty$ uniformly on compact sets and, according to assumption (\textit{A3}), we have
\[
\lim_{k\to \infty}  \frac{f'(z_k(x))}{f(z_k(x))} = \beta:= \lim_{t\to+\infty}{\frac{f'(t)}{f(t)}}\in [0,+\infty).
\]
Using \eqref{Delta2vk100}-\eqref{Delta2vk11} and the arbitrariness of $\varepsilon>0$, we find
\begin{equation*}\tag{*}\label{*}
\Delta^2v=a(x_\infty)e^{\beta v}\qquad\mbox{ in }\;\R^4.
\end{equation*}
Notice that $a(x_\infty)\neq0$ since $x_\infty\in\Omega$ and here $a>0$.  We now treat separately the subcritical case ($\beta=0$) and the critical case ($\beta>0$).

\subsection{The subcritical case ($\beta=0$)}\label{SectionSubcritical}
Here the limit profile $v$ satisfies
\begin{equation*}
\Delta^2v=a(x_\infty)\qquad\mbox{ in }\;\R^4,
\end{equation*}
with $a(x_\infty)>0$. Using a Taylor expansion and \eqref{asymptotic_GROWTH} with $\varepsilon=\frac12$, we have
\begin{equation*}\label{assurdo1}
\begin{split}
+\infty=\int_{\R^4} a(x_\infty)&=\int_{\R^4}\lim_{k\to+\infty}a(x_k+\mu_kx)e^{\frac{f'(z_k(x))}{f(z_k(x))}v_k(x)}\chi_{\Omega_k}(x)dx\\
&=\int_{\R^4}\lim_{k\to+\infty}a(x_k+\mu_kx)e^{\log(f(u_k(x_k+\mu_kx)))-\log(f(M_k))}\chi_{\Omega_k}(x)dx\\
&\leq 2\liminf_{k\to+\infty}\int_{\Omega_k}\dfrac{[\frac12a(x_k+\mu_kx)f(u_k(x_k+\mu_kx))-d]+d}{f(M_k)}dx\\
&\leq 2\liminf_{k\to+\infty}\bigg[\int_{\Omega_k}\dfrac{h(x_k+\mu_kx,u_k(x_k+\mu_kx))}{f(M_k)}dx+d\int_{\Omega_k}\dfrac{dx}{f(M_k)}\bigg]\\
&=2\liminf_{k\to+\infty}\bigg[\intOmega h(y,u_k(y))dy+d|\Omega|\bigg]\leq2[\Lambda+d|\Omega|],
\end{split}
\end{equation*}
where the last inequality is due to Proposition \ref{NEWSection1}. This contradiction proves Theorem \ref{uniformboundblowupBALL} in the case of a subcritical nonlinearity.

\subsection{The critical case ($\beta>0$)}\label{SectionQcritical}
 
First, with a similar argument as in Section \ref{SectionSubcritical}, we find that $v$ has finite energy, as a first step in order to characterize it.

\begin{lem}\label{integrability}
	$\int_{\R^4}e^{\beta v}<+\infty$.
\end{lem}
\begin{proof}
	\begin{equation*}
	\begin{split}
	a(x_\infty)\int_{\R^4}e^{\beta v}&=\int_{\R^4}\lim_{k\to+\infty}a(x_k+\mu_kx)e^{\frac{f'(z_k(x))}{f(z_k(x))}v_k(x)}\chi_{\Omega_k}(x)dx\\
	&\leq2\liminf_{k\rightarrow+\infty}\bigg[\intOmega h(y,u_k(y))dy+d|\Omega|\bigg]\leq2[\Lambda+d|\Omega|].
	\end{split}
	\end{equation*}
\end{proof}
	
\begin{lem}[\cite{MP}, Lemma 4]\label{MPlemma}
	For all $i=1,2,3$ and $p\in[1,\frac4i)$, there exists a constant $C(i,p)>0$ such that $\|\nabla^iu_k\|_{L^p(B_r(x_0))}^p\leq Cr^{4-ip}$ for any $B_r(x_0)\subset\Omega$.
\end{lem}
\begin{proof}
	\noindent By the Green representation formula and Lemma \ref{Greenspropetry} we have
	$$|\nabla^iu_k(x)|\leq\int_\Omega |\nabla_x^iG_\Omega(x,y)|h(y,u_k(y))dy\leq C\int_\Omega\dfrac{1}{|x-y|^i}h(y,u_k(y))dy.$$
	Thus, for any $\varphi\in C^{\infty}_c(B_r(x_0))$ and $p'$ being the conjugate exponent of $p$, we have
	\begin{equation*}
	\begin{split}
	\int_{B_r(x_0)}|\nabla^iu_k(x)|\varphi(x)dx
	&\leq C\int_\Omega h(y,u_k(y))\||x-y|^{-i}|\|_{\Lp(B_r(x_0))}\|\varphi\|_{L^{p'}(B_r(x_0))}dy\\
	&\leq\Lambda r^{4-ip}\|\varphi\|_{L^{p'}(B_r(x_0))},
	\end{split}
	\end{equation*}
	using Proposition \ref{NEWSection1} and the boundedness of $\Omega$. By duality, this yields our claim.
\end{proof}	

\begin{lem}
	$v(x)=o(|x|^2)$ as $|x|\rightarrow+\infty.$
\end{lem}
\begin{proof}
	Firstly, by \eqref{vk} and Lemma \ref{MPlemma} with $i=2$ and $p=1$, there holds
	\begin{equation}\label{MPstima}
	\begin{split}
	\int_{B_R(0)}|\Delta v_k|&=\mu_k^2\int_{B_R(0)}|\Delta u_k(x_k+\mu_kx)|dx=\mu_k^{-2}\int_{B_{R\mu_k}(x_k)}|\Delta u_k|\\
	&\leq\mu_k^{-2}C(R\mu_k)^{4-2}=CR^2.
	\end{split}
	\end{equation}
	\noindent Suppose now by contradiction that $v(x)=o(|x|^2)$ as $|x|\rightarrow+\infty$ does not hold. By Lin \cite{Lin} we would infer that there exists $b>0$ such that $-\Delta v(x)\geq b$ for every $x\in\R^4$. This, combined with \eqref{MPstima} and Fatou's Lemma, would imply
	\begin{equation*}
	CbR^4\leq\int_{B_{R}(0)}|\Delta v|\leq\liminf_{k\rightarrow+\infty}\int_{B_{R}(0)}|\Delta v_k|\leq CR^2,
	\end{equation*}
	which contradicts the arbitrariness of $R>0$. This proves our claim.	
\end{proof}

\noindent We can now characterize the limit profile $v$ and the energy concentrating at $x_\infty$ by means of the classification result by Lin \cite[Theorem 1.1]{Lin}.
\begin{lem}\label{charact} We have
\begin{equation}\label{exprv}
v(x)=-\frac4\beta\log\left(1+\left(\frac{a(x_\infty)\beta}{24}\right)^\frac12\frac{|x|^2}4\right).
\end{equation}
	Moreover, there holds
	\begin{equation}\label{stimaintegrale}
	\lim_{R\rightarrow+\infty}\liminf_{k\rightarrow+\infty}\int_{B_{R\mu_k}(x_k)}h(y,u_k(y))dy\geq \theta>0,
	\end{equation}
	with $\theta\beta=64\pi^2$.
\end{lem}
\begin{proof}
	Let us define $w= \frac{\beta}{4} v(\cdot T )$, where $T>0$ will be specified later. Then $w$ satisfies
	\begin{equation*}
		\Delta^2w=\frac{a(x_\infty)\beta T^4}4e^{4w}.
	\end{equation*}
	Choosing $T^4=\frac{24}{a(x_\infty)\beta}$, we have that $w$ is a finite-energy solution of $\Delta^2w=6e^{4w}$ such that $w(x)=o(|x|^2)$ as $|x|\to +\infty$.  By \cite[Theorem 1.1]{Lin}, there exist $\lambda>0$, $x_0\in \R^4$ such that $w = \eta_0(\lambda(\cdot-x_0))+\log \lambda$, where $\eta_0 = \log(\frac{2}{1+|x|^2})$. Since $w\le w(0)=0$, we get $x_0 =0$ and $\lambda = \frac{1}{2}$, so that $v$ is given by the expression in \eqref{exprv}.
	
	\noindent Let now $\varepsilon\in(0,1)$. Then, arguing as in Section \ref{SectionSubcritical} and using \eqref{asymptotic_GROWTH},
	\begin{equation*}
	\begin{split}
	0&<\theta:=a(x_\infty)\int_{\R^4}e^{\beta v}=\lim_{R\rightarrow+\infty}\int_{B_R(0)}\lim_{k\rightarrow+\infty}a(x_k+\mu_kx)e^{\big(\frac{f'(z_k(x))}{f(z_k(x))}\big)v_k(x)}dx\\
	&\leq\lim_{R\rightarrow+\infty}\liminf_{k\rightarrow+\infty}\int_{B_R(0)}a(x_k+\mu_kx)e^{\big(\frac{f'(z_k(x))}{f(z_k(x))}\big)v_k(x)}dx\\
	&=\frac1{1-\varepsilon}\lim_{R\rightarrow+\infty}\liminf_{k\rightarrow+\infty}\left(\int_{B_R(0)}\dfrac{(1-\varepsilon)a(x_k+\mu_kx)f(u_k(x_k+\mu_kx))-d_\varepsilon}{f(M_k)}dx+\dfrac{d_\varepsilon}{f(M_k)}|B_R(0)|\right)\\
	&\leq\frac1{1-\varepsilon}\lim_{R\rightarrow+\infty}\liminf_{k\rightarrow+\infty}\int_{B_{R\mu_k}(x_k)}h(y,u_k(y))dy.
	\end{split}
	\end{equation*}
  By arbitrariness of $\varepsilon$, this implies \eqref{stimaintegrale}. Finally, we observe that
  $$
  \theta=a(x_\infty)\int_{\R^4}e^{\beta v}dx=\frac{24}{\beta}\int_{\R^4}e^{4 w}dx = \frac{24}{\beta}\int_{\R^4}e^{4 \eta_0}dx =  \frac{24}{\beta}|\mathbb S^4| =\frac{64\pi^2}{\beta}.
  $$ 
\end{proof}

\begin{remark}
	The explicit value of the constant $\theta$ shows that the bound from below \eqref{stimaintegrale} of the concentration of the energy on blow-up points is indeed independent both on the potential $a(\cdot)$ and on the point $x_\infty$.
\end{remark}

So far, we have investigated the behaviour of each $u_k$ around a \textit{maximum} point $x_k$. This is indeed what happens for \textit{each} sequence of points $(y_k)_k$ such that $u_k(y_k)\nearrow+\infty$, as stated in the next Lemma. We do not include a proof, as this result can be obtained combining the argument in \cite[Claims 5-7]{RobWei} with our Lemmas \ref{grad_limitato}-\ref{charact} (see also \cite[Lemmas 7-8]{MP}):
\begin{lem}\label{integrability2}
	From the sequence $(u_k)_k$ one can extract a subsequence still denoted by $(u_k)_k$ for which the following holds.\\
	\noindent There are $P\in\N$ and converging sequences $\xki\to\XI$, $1\leq i\leq P$, with $u_k(\xki)\to+\infty$ such that, setting
	$$v_{k,i}(x):=u_k(\xki+\muki x)-u_k(\xki),\qquad\muki:=(f(u_k(\xki)))^{-1/4},$$
	we have
	\begin{enumerate}[(i)]
		\item $\frac{|\xki-x_{k,j}|}{\muki}\to+\infty$ as $k\to+\infty$ for $1\leq i\neq j\leq P$;
		\item $v_{k,i}\to v$ in $C^{3,\gamma}_{loc}(\R^4),$ for $1\leq i\leq P$, where $v$ is defined in Lemma \ref{charact} and estimate \eqref{stimaintegrale} holds;
		\item $\inf_{1\leq i\leq P}|x-\xki|^4h(x,u_k(x))\leq C$ for every $x\in\Omega$;
		\item $\inf_{1\leq i\leq P}|x-\xki|^j|\nabla^ju_k(x)|\leq C$ for every $x\in\Omega$ and $1\leq j\leq 3$.
	\end{enumerate}
\end{lem}

Denote finally by $S$ the set of blow-up points, namely 
$$
S:=\{y\in \bar \Omega\,|\,\exists (y_j)_j \subseteq \Omega \,|\,y_j\to y, u_{k_j}(y_j)\to+\infty \text{ as } j\to +\infty\}.
$$

Lemma \ref{integrability2} has two important consequences: the finiteness of the set $S$ and the local uniform boundedness in a strong norm for the sequence $(u_k)_k$ outside $S$ {(actually, for the subsequence we have extracted by Lemma \ref{integrability2})}.
\begin{prop}\label{Sfinite}
	The blow-up set $S$ is finite; moreover, one has
	\begin{equation}\label{localcompactness}
	\|u_k\|_{W^{3,\infty}_{loc}(\Omegabar\setminus S)}\leq C,
	\end{equation}
	for a uniform constant $C>0$.
\end{prop}
\begin{proof}
	We prove the finiteness of $S$ by showing that $S=\{\XI, 1\leq i\leq P\}$ defined in Lemma \ref{integrability2}. Indeed, suppose by contradiction that there exists $\bar{x}\not\in (\XI)_{i=1}^P$ and a sequence $\bar x_k\to\bar x$ such that $u_k(\bar x_k)\rightarrow +\infty$. Since $P<\infty$, one has $\inf_{k,i}|\bar x_k-\XI|\geq\bar d>0$. Notice also that $d(\bar x,\dOmega)\geq \eta>0$ by Proposition \ref{NEWSection1}, so $a(\bar{x}_k)\geq a_0>0$. Hence, by \textit{(iii)} of Lemma \ref{integrability2} and \eqref{asymptotic_GROWTH} with $\varepsilon=\frac12$, we get
	\begin{equation*}
	\begin{split}
	\frac{C}{\bar d^4}\geq h(\bar x_k,u_k(\bar x_k))\geq \frac12 a_0f(u_k(\bar x_k))-d,
	\end{split}
	\end{equation*}
	which in turn implies $u_k(\bar x_k)\leq C$ by the superlinearity of $f$.
	
	Let us now consider $K\subset\subset\Omegabar\setminus S$ and $r>0$ such that $K\cap B_r(\XI)=\emptyset$ for each $\XI\in S$. Firstly we know that $(u_k)_k$ is uniformly bounded on $K\cap\omega$ by Proposition \ref{NEWSection1}; in $K\setminus\,\omega$ we have $a(x)\geq a_0>0$ and, moreover, one may suppose $\inf_{1\leq i\leq P}|x-\xki|\geq \frac r2$ by construction. Hence, retracing the same argument above, we obtain also $\|u_k\|_{L^\infty(K\setminus\,\omega)}\leq C$. Finally, the boundedness of the derivatives of $(u_k)_k$ follows by \textit{(iv)} of Lemma \ref{integrability2}.
\end{proof}

So far, we found that there may be only a finite number of blow-up points and that near any of them the blow-up sequence $(u_k)_k$ has the same concentration of energy. We will now prove Theorem \ref{uniformboundblowupBALL} showing that this contradicts the bound on the energy in Proposition \ref{NEWSection1}.
\begin{proof}[Proof of the Theorem \ref{uniformboundblowupBALL}]
By \eqref{maxima}, we know that $S \neq \emptyset$.	Let $x_0\in S$ and consider $r>0$ small enough such that $B_r(x_0)\cap S=\{x_0\}$. Here, define $h_k:=h(x,u_k(x))$ and let $\phi_k$ and $v_k$ be the solutions of the problems
	\begin{equation*}
	\begin{cases}
		\Delta^2\phi_k=h_k\,\,&\mbox{in }B_r(x_0),\\
		\phi_k=\Delta\phi_k=0\,\,&\mbox{on }\partial B_r(x_0),
	\end{cases}\qquad\qquad
	\begin{cases}
		\Delta^2w_k=0\,\,&\mbox{in }B_r(x_0),\\
		w_k=u_k,\;\Delta w_k=\Delta u_k\,\,&\mbox{on }\partial B_r(x_0).                         
	\end{cases}
	\end{equation*}
	Readily we have $u_k=\phi_k+w_k$; moreover, by Proposition \ref{Sfinite} and elliptic estimates (see for instance \cite[Theorem 2.20]{GGS}) it is immediate to infer $\|w_k\|_\infty\leq C$. Furthermore, $\|h_k\|_{L^1(B_r(x_0))}\leq \Lambda$ by Proposition \ref{NEWSection1}, so $h_k\rightharpoonup\mu$, where $\mu\in\mathcal M(B_r(x_0))\cap L^{\infty}_{loc}(B_r(x_0)\setminus\{x_0\})$ is a positive measure.
		
	We first claim that
	\begin{equation}\label{mudelta}
	\mu\geq\theta\delta_{x_0},
	\end{equation}
	where $\delta_{x_0}$ denotes the Dirac distribution centered in $x_0$. Indeed, for any $t\in(0,r)$,
	\begin{equation*}
		\int_{B_t(x_0)}d\mu=\lim_{k\to+\infty}\int_{B_t(x_0)}h(x,u_k(x))dx\geq\liminf_{k\to+\infty}\int_{B_{R\mu_k}(x_k)}h(x,u_k(x))dx\geq\theta
	\end{equation*}
	by Lemma \ref{charact}, where $x_k$ is a blow-up sequence converging to $x_0$.
	
	Moreover, let $\phi$ the distributional solution of
	\begin{equation*}
	\begin{cases}
		\Delta^2\phi=\mu\,\,&\mbox{in }B_r(x_0),\\
		\phi=\Delta\phi=0\,\,&\mbox{on }\partial B_r(x_0).
	\end{cases}
	\end{equation*}
	As $h_k\rightharpoonup\mu$, one may decouple the Navier problem into a system of second-order Dirichlet problems and prove that $\phi_k\to\phi$ in $W^{3,q}(B_r(x_0))$ for $q\in[1,\frac43)$ by means of standard elliptic theory, together with regularity theory for second-order Dirichlet problems with measure data, see for instance \cite[Proposition 5.1]{Ponce}. Furthermore, let $G_{NAV,B_r(x_0)}$ be the Green function of the biharmonic operator with Navier boundary conditions in the ball $B_r(x_0)$. Then we know that $G_{NAV,B_r(x_0)}(\cdot,x_0)\geq0$ solves
	\begin{equation}\label{Green_Nav}
	\begin{cases}
		\Delta^2G_{NAV,B_r(x_0)}(\cdot,x_0)=\delta_{x_0}\,\,&\mbox{in }B_r(x_0),\\
		G_{NAV,B_r(x_0)}(\cdot,x_0)=\Delta G_{NAV,B_r(x_0)}(\cdot,x_0)=0\,\,&\mbox{on }\partial B_r(x_0)
	\end{cases}
	\end{equation}
	and therefore, by \eqref{mudelta} and the maximum principle in \cite[Proposition 6.1]{Ponce},
	\begin{equation}\label{fund_estimate}
		\phi\geq\theta G_{NAV,B_r(x_0)}(\cdot,x_0)
	\end{equation}
	in the distributional sense and almost everywhere in $B_r(x_0)$.
		
	Let now $\varepsilon\in(0,\beta)$ be small. Then by \eqref{fund_estimate},
	\begin{equation*}
	\begin{split}
		\int_{B_r(x_0)}e^{(\beta-\varepsilon)\phi}dx&\geq\int_{B_r(x_0)}e^{(\beta-\varepsilon)\theta G_{NAV,B_r(x_0)}(x,x_0)}.
	\end{split}
	\end{equation*}
	Since the leading term of the Green function in a ball is the fundamental solution, namely
	$$G_{NAV, B_r(x_0)}(x,x_0)=\frac1{8\pi^2}\log\left(\frac1{|x-x_0|}\right)+H_{x_0}(x),$$
	where $H_{x_0}(\cdot)$ a regular function in $\overline{B_r(x_0)}$, then by Lemma \ref{charact} we get
	\begin{equation}\label{contradiction_1}
	\begin{split}
		\int_{B_r(x_0)}e^{(\beta-\varepsilon)\phi}dx&\geq\int_{B_r(x_0)}Ce^{\frac{(\beta-\varepsilon)\theta}{8\pi^2}\log\left(\frac1{|x-x_0|}\right)}dx=C\int_{B_r(x_0)}\left(\frac1{|x-x_0|}\right)^{\frac{(\beta-\varepsilon)\theta}{8\pi^2}}dx\\
		&\geq C\int_{B_r(x_0)}\left(\frac1{|x-x_0|}\right)^{8-\varepsilon'}dx=+\infty
	\end{split}
	\end{equation}
	for $\varepsilon'=\frac{\theta \varepsilon}{8\pi^2}$ sufficiently small. On the other hand, using the decomposition of $u_k$ by means of $\phi_k$ and $w_k$, the boundedness of $w_k$, Lemma \ref{fexpbeta_lem}, \eqref{asymptotic_GROWTH} and Proposition \ref{NEWSection1}, we finally have
	\begin{equation*}
	\begin{split}
		\int_{B_r(x_0)}e^{(\beta-\varepsilon)\phi}dx\,&\leq \lim_{k\to+\infty}\int_{B_r(x_0)}e^{(\beta-\varepsilon)(u_k-w_k)}dx\leq C\lim_{k\to+\infty}\int_{B_r(x_0)}f(u_k(x))dx\\&\leq C_1\lim_{k\to+\infty}\int_{B_r(x_0)}h(x,u_k(x))dx+C_2\leq C(\Lambda).
	\end{split}
	\end{equation*}
	This contradicts \eqref{contradiction_1} and concludes the proof.
	\end{proof}

\section{Some extensions of Theorem \ref{uniformboundblowupBALL}}

\subsection{Extension to general smooth domains}\label{SectionDomains}

Theorem \ref{uniformboundblowupBALL} deals with a large range of nonlinearities, but it applies uniquely when the domain is a ball. The restrictions to its extension to more general domains are mainly two. Firstly, unless the domain is positivity preserving, we cannot ensure that solutions of \eqref{DIRf}, as well as the first eigenfunction $\tilde\varphi_1$, are positive, and all estimates of Section \ref{SectionNaPR} rely on this fact.
Secondly, the two-sided estimate \eqref{2sided} is available only for balls, so we are able to start the moving-planes procedure and control solutions near the boundary as in Lemma \ref{boundaryestDIR} only in this case. Nevertheless, once we consider solutions for which the estimate $\intOmega h(x,u)dx\leq\Lambda$ holds a-priori, a careful reading of Sections \ref{SectionSubcritical} and \ref{SectionQcritical} shows that the blow-up argument applies to \textit{any} domain, provided all blow-up points lye inside the domain. Here, we show how to deal with blow-up at the boundary and we prove Theorem  \ref{uniformboundblowupOMEGA}.

Notice that, in order to exclude concentration near a point $x_\infty\in\dOmega$, we have to impose strict positivity on the coefficient $a$ to be sure that $a(x_\infty)>0$. Moreover, it is not hard to prove by (\textit{H2}) and de l'H\^{o}pital's Theorem that		\begin{equation}\label{hopital}
	\lim_{t\to+\infty}\,\dfrac{H(x,t)}{a(x)F(t)}=\lim_{t\to+\infty}\,\dfrac{h(x,t)}{a(x)f(t)}=1
\end{equation}
uniformly with respect to $x\in\Omega$.
\begin{proof}[Proof of Theorem \ref{uniformboundblowupOMEGA}] 
	By contradiction, suppose the existence of points $(x_k)_{k\in\N}\subset\Omega$ and solutions $(u_k)_{k\in\N}\subset\Hoo$ such that $M_k:=\|u_k\|_{L^\infty(\Omega)}=u_k(x_k)\nearrow+\infty$. We define $\Omega_k$ and $\mu_k$ as in \eqref{muk} and the rescaled functions $v_k$ as in \eqref{vk}. Since $\Omega$ is bounded, the maximum points $x_k$ accumulate at some $x_\infty\in\Omegabar$. We claim that, in any case, $\frac{d(x_k,\dOmega)}{\mu_k}\to+\infty$, so that $\Omega_k\nearrow\R^4$. Indeed, suppose by contradiction that $d(x_k,\dOmega)=O(\mu_k)$, that is, up to an affine transformation, $\Omega_k\rightarrow(-\infty,0)\times\R^3$. Letting $R>0$ and $x\in B_R(0)\cap\Omegabar_k$, with the same computations of Lemma \ref{grad_limitato}, we infer $|\nabla^i v_k|\leq C(R)$ for any $x\in B_R(0)\cap\Omegabar_k$. Choosing $x\in B_R(0)\cap\dOmega_k$, so that $v_k(x)=-M_k$, we would get
	\begin{equation*}
	M_k=|v_k(x)|=|v_k(x)-v_k(0)|\leq CR,
	\end{equation*}
	a contradiction. Then, the same compactness argument as in Section \ref{SectionBlowup} proves that $v_k\to v$ locally uniformly as $k\to +\infty$, where $v$ is a solution of $\Delta^2v=a(x_\infty)e^{\beta v}$ in $\R^4$.  For the subcritical case \textit{(1)} the a-priori bound is obtained with the same argument as in Section \ref{SectionSubcritical}. 
	
	Let us suppose now assumption \textit{(2)}.  In this case, the analogues of Lemmas \ref{integrability}-\ref{integrability2} and Proposition \ref{Sfinite} hold, so in particular the blow-up set $S$ of $u_k$ is finite and we have the concentration of the energy near the points of $S$. However, the final argument in Section \ref{SectionBlowup} permits to exclude only the existence of blow-up points inside $\Omega$, but not those lying on $\dOmega$. To rule out also this possibility, firstly we study the global limit profile of the sequence $(u_k)_k$ and then, in the spirit of \cite{RobWei}, we apply a Poho\v{z}aev-type identity.

	\begin{lem}\label{lemma_limits}
		Let $h_k(x):=h(x,u_k(x))$. Then, {up to a subsequence,}
		\begin{equation}\label{RHSlimit}
		h_k\rightharpoonup r(x)dx+\sum_{p\in S}m_p\delta_p,
		\end{equation}
		where $m_p$ are nonnegative constants, $0\leq r\in L^1(\Omega)$ and $\delta_p$ are Delta distributions centered in $p$. Moreover,
		\begin{equation}\label{uk_limit}
		u_k\to U
		\end{equation}
		in $W^{2,q}_0(\Omega)\cap C^{3,\alpha}_{loc}(\Omegabar\setminus S)$ for $1\leq q<2$, where $U$ solves
		\begin{equation}\label{U_Gp}
		\begin{cases}
		\Delta^2U=r(x)\,\,&\mbox{in }\Omega,\\
		U=U_n=0\,\,&\mbox{on }\dOmega.
		\end{cases}
		\end{equation}
	\end{lem}
	\begin{proof}
Let $K\subset\subset\Omegabar\setminus S$, so $\|u_k\|_{L^\infty(K)}\leq C(K)$ by Proposition \ref{Sfinite} and in turn $\|h_k\|_{L^\infty(K)}\leq C(K)$. This means that the singular part of the limit lies in $S$, namely \eqref{RHSlimit} holds for $r\in L^1(\Omega)\cap L_{loc}^\infty(\Omegabar\setminus S)$. Moreover, let $U$ be as in \eqref{U_Gp} and $\varphi\in C^\infty_c(\Omega)$. Then, since no interior blow-up occurs, we have
		\begin{equation}
		\intOmega\Delta u_k\Delta\varphi=\intOmega h(x,u_k(x))\varphi(x)dx\rightarrow\intOmega r(x)\varphi(x)dx
		=\intOmega\Delta U\Delta\varphi.
		\end{equation}
		On the other hand, as $\|h_k\|_{L^1(\Omega)}\leq\Lambda$ by assumption, then $u_k\to\tilde u$ in $W^{2,q}_0(\Omega)$ for any $1\leq q<2$ and
		$$\intOmega\Delta u_k\Delta\varphi\to\intOmega\Delta\tilde u\Delta\varphi.$$
		By uniqueness of the limit, \eqref{uk_limit} is proved.
	\end{proof}

The next lemma clarifies the behaviour of $U$ near the points of $S$. The proof is inspired by \cite[Lemma 3.7]{JY}. 
\begin{lem}\label{LemmabehaviourU}
For any $\varepsilon >0$, there exists $r_\varepsilon>0$ such that the following holds: for any $x_0\in S$ and any $x\in \Omega$ s.t. $0<|x-x_0|< r_\varepsilon$ we have
\begin{equation}\label{estU}
|U|\le \varepsilon \log\left(2+\frac{1}{|x-x_0|}\right)
\end{equation}
and
\begin{equation}\label{estnabla^iU}
|\nabla^i U|\le \frac{\varepsilon}{|x-x_0|^i},
\end{equation}
for $i=1,2,3$. 
\end{lem}
\begin{proof}
Let us denote $f_i(t):= \frac{1}{t^i}$ for $i=1,2,3$ and $f_0(t)=\log(2+\frac{1}{t})$. By Lemma \ref{Greenspropetry} and part (\textit{iv}) of Lemma \ref{integrability2}, there exists a constant $\bar C>0$ such that 
 $|\nabla^i G(x,y)|\le \bar C f_i(|x-y|)$ and 
 \begin{equation}\label{boundonr}
\max_{y\in \Omega}\big(r(y)\min_{p\in S} |y-p|^4\big)\le \bar C.
 \end{equation} 
 Moreover, we can assume that 
 \begin{equation}\label{furthercondition}
\int_{B_{s}(0)}f_i(|x|)dx \le \bar C s^4 f_i(s),
\end{equation}
for any $s\in (0,1)$. Let us choose $\sigma=\sigma(\varepsilon)\in (0,1)$ such that $$\frac{\bar C^3 \sigma}{(1-\sigma)^4} \le \frac{\varepsilon}{4}\quad\mbox{and}\quad\sigma^3f_0(\sigma)\le\log 2.$$
Then, let $\delta= \delta(\varepsilon)<\frac{1}{4} \min_{p,q\in S} |p-q|$, be such that $\frac{\bar C}{\sigma^3}\|r\|_{L^1(B_{2\delta}(x_0))} \le \frac{\varepsilon}{4}$. Let us now take $x_0\in S$ and $x\in \Omega$ with $|x-x_0|<\delta$. Using the Green representation formula, we have
$$
|\nabla^i U| \le \bar C \int_{\Omega} f_i(|x-y|)r(y) dy. 
$$ 
We split the integral into three regions:
$$\Omega_1:= \Omega\setminus B_{2\delta(x_0)},\quad\Omega_2:=\Omega\cap B_{2\delta} (x_0)\cap B_{\sigma |x-x_0|}(x),\quad\Omega_3:=\Omega \cap(B_{2\delta} (x_0)\setminus B_{\sigma |x-x_0|}(x)).$$
First, for any $y\in \Omega_1$, we have $|x-y|\ge |y-x_0|-|x-x_0|\ge  \delta$, so that 
\begin{equation}\label{omega1}
\int_{\Omega_1} f_i(|x-y|)r(y) dx  \le  f_i(\delta) \|r\|_{L^1(\Omega)}.
\end{equation}
For $y\in \Omega_2$, we have $|y-x_0|<|y-x|+|x-x_0|\le (1+\sigma)|x-x_0|< \frac{1}{2}d(y,S\setminus \{x_0\})$  and  $|y-x_0|\ge |x-x_0|-|y-x|\ge (1-\sigma)r$. Then, thanks to \eqref{boundonr} and \eqref{furthercondition} and our choice of $\sigma$, we get 
\begin{equation}\label{omega2}
\begin{split}
\int_{\Omega_2} f_i(|x-y|)r(y)dy &\le \bar C\int_{\Omega_2} \frac{f_i(|x-y|)}{|y-x_0|4} dy \le \frac{\bar C}{(1-\sigma)^4|x-x_0|^4} \int_{B_{\sigma|x-x_0|}(x)} f_i(|x-y|)dy \\
&\le \frac{\bar C^2\sigma^4}{(1-\sigma)^4} f_i(\sigma |x-x_0|) \le  \frac{\varepsilon \sigma^3}{4\bar C} f_i(\sigma |x-x_0|). 
\end{split}
\end{equation}
Finally, on $\Omega_3$ we have the estimate
\begin{equation}\label{omega3}
\int_{\Omega_3} f_i(|x-y|) r(y) dy \le f_i(\sigma|x-x_0|)\int_{B_{2\delta}(x_0)} r(y) dy \le  \frac{\varepsilon \sigma^3}{4\bar C} f_i(\sigma|x-x_0|).
\end{equation}
where the last inequality follows from the choice of $\delta$. Observe that our choice of $\sigma$ guarantees that  $\sigma^3 f_i(\sigma|x-x_0|)\le f_i(|x-x_0|)$ for $i=0,1,2,3$. Using \eqref{omega1}-\eqref{omega3} we can conclude that there exists $K=K(\varepsilon)$ such that
$$
|\nabla^i U(x)|\le K + \frac{\varepsilon}{2} f_i(|x-x_0|).
$$
Then, it is sufficient to choose $r_\varepsilon\in (0,\delta)$ such $K\le \frac{\varepsilon}{2} f_i(r_\varepsilon)$.  
\end{proof}

The following version of the Poho\v{z}aev identity can be proved as in \cite{RobWei} Lemma 2.2] with only minor modifications (see also \cite{Mit}).
	
	\begin{lem}\label{MitMOD}
		Let $u\in H^4(\Omega)$ be a strong solution of $\Delta^2u=h(x,u)$ in $\Omega$. Then, for any $y\in\R^4$, we have
		\begin{equation*}
			4\int_\Omega H(x,u)dx+\int_\Omega\langle x-y,\nabla_xH(x,u)\rangle dx=\int_\dOmega\langle x-y,n(x)\rangle H(x,u)d\sigma+b(y,u),
		\end{equation*}
		where $H$ is defined in \eqref{F} and $b$ collects all remaining boundary terms:
		\begin{equation*}
			\begin{split}
				b(y,u):=&\;\dfrac12\int_\dOmega(\Delta u)^2\langle x-y,n(x)\rangle d\sigma-2\int_\dOmega u_n\Delta ud\sigma-\int_\dOmega(\Delta u)_n\langle x-y,\nabla u\rangle d\sigma\\
				&-\int_\dOmega u_n\langle x-y,\nabla(\Delta u)\rangle d\sigma+\int_\dOmega\langle\nabla(\Delta u),\nabla u\rangle\langle x-y,n(x)\rangle d\sigma.
			\end{split}
		\end{equation*}
	\end{lem}

		Let us return to the proof of Theorem \ref{uniformboundblowupOMEGA}. Let $x_0\in S\subset\dOmega$, $r>0$ be sufficiently small such that $B_r(x_0)\cap S=\{x_0\}$ and apply the identity of Lemma \ref{MitMOD} in $\Omega\cap B_r(x_0)$ to $u_k\in H^4(\Omega)$ by elliptic regularity. Notice that we have to deal with two kinds of boundary terms, the ones relative to $\Omega\cap\partial B_r(x_0)$ and the others to $\dOmega\cap B_r(x_0)$. Notice also that for any choice $y$ all boundary terms on $\dOmega\cap B_r(x_0)$ involving derivatives of $u_k$ vanish by the boundary conditions satisfied by $u_k$ except for the term $\int_{\dOmega\cap B_r(x_0)}(\Delta u_k)^2\langle x-y,n(x)\rangle d\sigma$. Therefore, we have to choose in a clever way a sequence of points $(y_k)_k$ so that this term vanishes too. Following \cite{RobWei}, we define $y_k:=x_0+\rho_{k,r}n(x_0)$, where
		\begin{equation*}
		\rho_{k,r}:=\dfrac{\int_{\dOmega\cap B_r(x_0)}(\Delta u_k)^2\langle x-x_0,n(x)\rangle dx}{\int_{\dOmega\cap B_r(x_0)}(\Delta u_k)^2\langle n(x_0),n(x)\rangle dx}
		\end{equation*}
		and, up to a smaller value, we may choose $r$ so that $\frac12\leq\langle n(x_0),n(x)\rangle\leq1$ for all $x\in\overline B_r(x_0)\cap\Omega$. With these choices, we have $|\rho_{k,r}|\leq2r$ and
		\begin{equation*}
		\int_{\dOmega\cap B_r(x_0)}(\Delta u_k)^2\langle x-y_k,n(x)\rangle dx=0.
		\end{equation*}
		Our aim is to prove that, while the terms on the left-hand side of the Poho\v{z}aev identity are bounded from below by a positive constant thanks to the concentration of the energy near the blow-up points, on the other hand all the remaining boundary terms on the right-hand side can be made arbitrary small as $r\to 0$.
		
		Let us start from the left-hand side: \eqref{hopital} and the assumption \eqref{hpnablaH} imply
		\begin{equation*}
			\begin{split}
				&4\int_{\Omega\cap B_r(x_0)} H(x,u_k)dx+\int_{\Omega\cap B_r(x_0)}\langle x-y_k,\nabla_xH(x,u_k)\rangle dx\\
				&\geq2\int_{\Omega\cap B_r(x_0)} a(x)F(u_k)dx-4d |B_r(x_0)|-\int_{\Omega\cap B_r(x_0)}3rB(x)F(u_k)dx -3r\int_{\Omega\cap B_r(x_0)}D(x)dx \\
				&\geq \left(a(x_0)-3r\|B\|_{L^\infty(\omega)}\right)\int_{\Omega\cap B_r(x_0)}F(u_k)dx-d Cr^4-3r\int_{\omega}D(x)dx ,
			\end{split}
		\end{equation*}
		when $r>0$ is so small that $a(x)>\frac{a(x_0)}{2}$ for any $x\in B_r(x_0)$. Supposing further, up to a smaller value of $r$ that $3r\|B\|_{L^\infty(\omega)}\leq\frac{a(x_0)}{2}$, we find
		\begin{equation}\label{LHS}
			4\int_{\Omega\cap B_r(x_0)} H(x,u_k)dx+\int_{\Omega\cap B_r(x_0)}\langle x-y_k,\nabla_xH(x,u_k)\rangle dx+o(r)\geq\frac{a(x_0)}{2}\int_{\Omega\cap B_r(x_0)}F(u_k)dx.
		\end{equation}
		We now claim there exists $m\geq1$ such that
		\begin{equation}\label{riducodim}
			F(t)\geq\dfrac1m \big(f(t)-f(0)\big)\qquad\mbox{for any}\;\;t\geq0.
		\end{equation}
		Indeed, let us define
		$$m:=\max\bigg\{1,\max_{t\geq0}\dfrac{f'(t)}{f(t)}\bigg\}.$$
		Then $1\leq m<+\infty$ since $0<f\in C^1([0,+\infty))$ by (\textit{A1}), together with the limit assumption (\textit{A3}). Hence
		$$\bigg(F-\dfrac1mf\bigg)'(t)=f(t)-\dfrac1mf'(t)=f(t)\bigg(1-\dfrac1m\dfrac{f'(t)}{f(t)}\bigg)\geq 0,$$
		which implies \eqref{riducodim}.
		Finally, we further estimate from below the left-hand side of \eqref{LHS}, by \eqref{riducodim} and Lemma \ref{charact}, obtaining
		\begin{equation}\label{LHSPoho}
			\begin{split}
				\frac{a(x_0)}{2}\int_{\Omega\cap B_r(x_0)} F(u_k)&\geq \frac{a(x_0)}{2m}\int_{\Omega\cap B_r(x_0)}f(u_k) + o(r)\\
				&\geq\frac{a(x_0)}{4m\|a\|_\infty}\bigg[\int_{\Omega\cap B_r(x_0)}(2a(x)f(u_k)+d)-d|B_r(x_0)|\bigg] +o(r)\\
				&\geq\frac{a(x_0)}{4m\|a\|_\infty}\int_{\Omega\cap B_r(x_0)}h(x,u_k(x))dx + o(r)\\
				&\geq\frac{a(x_0)}{4m\|a\|_\infty}\dfrac\theta 4 +o(r).
			\end{split}
		\end{equation}

		Let us now prove that the boundary terms in the right-hand side of the Poho\v{z}aev identity may be arbitrary small as $r\to 0$. We recall that we only have to analyse those on $\Omega\cap \partial B_r(x_0)$. To this aim, let $U$ as in Lemmas \ref{lemma_limits}-\ref{LemmabehaviourU}.

		Since $|x-y_k|\leq|x-x_0|+|\rho_{k,r}|\leq 3r$, by means of Lemma \ref{lemma_limits} and \eqref{estnabla^iU}, we obtain that 
		\begin{equation*}
		\begin{split}
		\left|\int_{\Omega\cap\partial B_r(x_0)}(\Delta u_k)^2\langle x-y_k,n(x)\rangle d\sigma\right|&\leq o(1)+3r\int_{\Omega\cap\partial B_r(x_0)}|\Delta U|^2 d\sigma,\\
		&\leq o(1)+C\varepsilon^2
		\end{split}
		\end{equation*}
		as $r\to 0$. In a similar way, one can estimate the other boundary terms in $b(y_k,u_k)$.	
     	It remains now to control the term
		$$\int_{\Omega\cap\partial B_r(x_0)}\langle x-y_k,n(x)\rangle H(x,u_k)d\sigma.$$
		From \eqref{hopital} and simply integrating term by term the inequalities \eqref{fexpbeta} we get
		$$H(x,t)\leq2\|a\|_\infty F(t)+D\leq C(e^{(\beta+1)t}+1).$$
		Moreover, by Lemma \ref{lemma_limits} and \eqref{estnabla^iU}
		\begin{equation*}
		H(x,u_k(x))\leq Ce^{(\beta+1)U}+C \le C\bigg( \bigg(2+\frac{1}{r}\bigg)^{\varepsilon(\beta+1)}+1 \bigg).
		\end{equation*}
		Therefore, if $\varepsilon<\frac4{\beta+1}$,we obtain
		\begin{equation*}
		\begin{split}
		\bigg|\int_{\Omega\cap\partial B_r(x_0)}&\langle x-y_k,n(x)\rangle H(x,u_k)d\sigma \bigg|
		\leq 3 r \int_{\Omega\cap\partial B_r(x_0)} H(x,u_k)d\sigma \to 0.
		\end{split}
		\end{equation*}
		In conclusion, we proved that all boundary terms of Poho\v{z}aev identity can be made arbitrary small as $r\to 0$, and this contradicts \eqref{LHS}-\eqref{LHSPoho}. 
		\end{proof}

\subsection{Extension to the polyharmonic case}\label{SectionPOLY}

By now, we proved the a-priori bound for problem \eqref{DIRf_generalpoly} when $m=2$. In this section we show that the same strategy applies also in the general polyharmonic case:
\begin{equation}\label{polyDIRa}
\begin{cases}
(-\Delta)^mu=h(x,u)\quad&\mbox{in }\Omega\subset\R^{2m},\,\,m\geq1,\\
u=\partial_nu=\cdots=\partial_n^{m-1}u=0\quad&\mbox{on }\dOmega,
\end{cases}
\end{equation}
where $\Omega$ is a bounded smooth domain, that is, of class $C^{2m,\gamma}$, for some $\gamma\in(0,1)$, and provided the same assumptions on $h$ are satisfied. Herein, by \textit{weak solution} of \eqref{polyDIRa} we mean a function $u\in H^m_0(\Omega)$ such that
$$\int_\Omega\nabla^mu\nabla^m\varphi=\intOmega h(x,u)\varphi,$$
for every $\varphi\in H^m_0(\Omega)$, with the convention
\begin{equation*}
\nabla^m :=\quad\left\{ \begin{array}{c l}\displaystyle
\Delta^{m/2},\quad&\mbox{$m$ odd,}\\
\nabla\Delta^{(m-1)/2},\quad&\mbox{$m$ even.}                            
\end{array}\right.
\end{equation*}
\begin{remark}
In the second-order case $m=1$ and when $\Omega$ is a ball, our result extends the analysis in \cite{dFLN,GS,BM,ChenLi}, enlarging the class of nonlinearities for which the uniform boundedness of solutions holds. Actually, one can show that an analogue result holds also for more general domains, see \cite{MIO_N}.
\end{remark}

\begin{thm}
	Let $\B$ be a ball in $\R^{2m}$ and $h$ satisfy (\textit{H1})-(\textit{H3}) and one of conditions (i) or (ii) of Theorem \ref{uniformboundblowupBALL}.
	Then there exists $C>0$ such that $\|u\|_{L^\infty(\B)}\leq C$ for all weak solutions of \eqref{polyDIRa}.
\end{thm}
\begin{proof}
	An analogue of Proposition \ref{NEWSection1} in dimension $2m$ holds, following the same arguments of Section \ref{SectionNaPR} with only evident changes. In particular, Lemma \ref{lemmaDir1-B}  holds by considering $\lambda_1^{(m)}$ and $\varphi_1^{(m)}$ as respectively the first eigenvalue and the first eigenfunction in $\Omega$ of the operator $(-\Delta)^m$ subjected to Dirichlet boundary conditions. Moreover, the moving-planes argument of Lemma \ref{MovingPlanes} still applies, since the Green function vanishes near the boundary precisely of order $m$, 
	namely (see \cite[Theorem 4.6]{GGS})
	\begin{equation*}
	G_{(-\Delta)^m,\B}(x,y)\simeq \log\bigg(1+\dfrac{d_\B(x)^md_\B(y)^m}{|x-y|^{2m}}\bigg)
	\end{equation*}
	and we have the same estimates of Lemmas \ref{BGW_lemma} and \ref{FGW_lemma}. Then, Lemmas \ref{boundaryestDIR} and \ref{uniformLambda} easily follow, since they rely on properties that do not depend on the differential operator. Let us focus now on the blow-up argument. A careful reading of Section \ref{SectionBlowup} may convince the reader that all statements are still valid for problem \eqref{polyDIRa} once we adapt the scaling as
	\begin{equation*}%
	\mu_k^{(m)}:=(f(M_k))^{-\frac1{2m}}.
	\end{equation*}
	In particular, Lemma \ref{grad_limitato} holds for any $i\in\{0,1,\cdots, 2m-1\}$. Hence, again by the local compactness guaranteed by Lemma \ref{RWthm}, we find $v\in C^{2m-1}(\R^{2m})$ satisfying 
	\begin{equation}\tag{*$_m$}\label{*m}
	(-\Delta)^mv=a(x_\infty)e^{\beta v}\quad\mbox{in}\;\R^{2m}.
	\end{equation} In the subcritical case, one has $\beta=0$, an the contradiction is found exactly as in Section \ref{SectionSubcritical}. On the other hand, in the critical framework $\beta\in(0,+\infty)$, the characterization of entire solutions of equation \eqref{*m} follows by the classification result of Martinazzi \cite[Theorem 2]{Mart}. The value of the constant $\theta$ in \eqref{stimaintegrale} is now $\theta=\frac{(2m)!}{\beta}|\mathbb S^{2m}|$. Therefore, as again the leading term of the Green function of the polyharmonic extension of the Navier problem \eqref{Green_Nav} is $\frac1{\gamma_m}\log\left(\frac1{|x-x_0|}\right)$ with $$\gamma_m=|\mathbb S^{2m-1}|2^{2m-2}[(m-1)!]^2$$
	(see for instance \cite[Proposition 22]{Mart}), taking into account that
	$$|\mathbb S^{2m-1}|=\frac{2\pi^m}{(2m-1)!}\qquad\mbox{and}\qquad|\mathbb S^{2m}|=\frac{2^{m+1}\pi^m}{(2m-1)!!},$$
	then one infers that the crucial exponent in the contradictory argument is $\frac{\beta\theta}{\gamma_m}=4m>2m$, so the proof can be concluded as for the biharmonic case.
\end{proof}

Finally, also the higher-order counterpart of Theorem \ref{uniformboundblowupOMEGA} can be established following the analysis in Section \ref{SectionDomains} with the appropriate adaptations mentioned above. We just recall that in this context the Poho\v{z}aev identity assumes the following form: for any $u\in C^{2m}(\Omegabar)$ solution of $(-\Delta)^mu=h(x,u)$ in $\Omega\subset\R^{2m}$ and for any $y\in\R^{2m}$, there holds:
	\begin{equation*}
	\begin{split}
		& 2m\int_\Omega H(x,u)dx+\int_\Omega\langle x-y,\nabla_xH(x,u)\rangle dx\\
		&=\int_\dOmega\langle x-y,n(x)\rangle H(x,u)d\sigma+\dfrac12\int_\dOmega(\nabla^mu)^2\langle x-y,n(x)\rangle d\sigma\\
		&+\sum_{j=0}^{m-1}(-1)^{m+j}\int_\dOmega\langle n(x),\nabla^j((x-y)\nabla u)\nabla^{2m-1-j}u(x)\rangle d\sigma.
	\end{split}
	\end{equation*}
	
\begin{thm}
	Let $\Omega\subset\R^{2m}$ be a bounded $C^{2m,\gamma}$ smooth domain, $\gamma\in(0,1)$ and $h$ be a nonlinearity satisfying assumptions (\textit{H1})-(\textit{H2}) with $0<a_0\leq a(\cdot)\in C(\Omegabar)$. Suppose further that one of the conditions \textit{(1)}-\textit{(2)} of Theorem \ref{uniformboundblowupOMEGA} is satisfied. Then, for any $\Lambda>0$ there exists $C>0$ depending on $\Lambda$ such that  $\|u\|_{L^\infty(\Omega)}\leq C$ for all solutions of \eqref{DIRf} such that $\int_\Omega h(x,u)dx\leq\Lambda$.
\end{thm}

\subsection{The Navier boundary conditions}

So far, we studied problems endowed with Dirichlet boundary conditions. In this section, we show that a similar analysis can be also given when considering \textit{Navier boundary conditions}:
\begin{equation}\label{NAVf}
\begin{cases}
\Delta^2u=h(x,u)\quad&\mbox{in }\Omega,\\%
u=\Delta u=0\quad&\mbox{on }\dOmega.
\end{cases}
\end{equation}

\noindent Here we assume $\Omega\subset\R^4$ is a bounded \textit{convex} $C^{2,\gamma}$ domain and $h:\Omega\times\R\to\R^+$ satisfies (\textit{H1})-(\textit{H2}). Recall that $u\in H^2(\Omega)\cap\Ho$ is a \textit{weak solution} of \eqref{NAVf} whenever \eqref{def} holds for any test function $\varphi\in H^2(\Omega)\cap\Ho$.
\vskip0.2truecm
The proof of the analogue of Theorem \ref{uniformboundblowupBALL} for the Navier problem \eqref{NAVf} (the forthcoming Theorem \ref{NAV_uniformbound}), essentially follows the arguments contained in Sections \ref{SectionNaPR} and \ref{SectionBlowup}. We mention here only the remarkable modifications.\\
First, \eqref{NAVf} can be split in a system of second-order Dirichlet problems, so the maximum principle holds. The starting local $L^1$ estimate \eqref{localL1bdd} may be obtained retracing the proof of Lemma \ref{lemmaDir1-B}, where the Green function estimates therein are replaced by the analogues for the Navier boundary conditions (see \cite[Proposition 4.13]{GGS}). Then, the argument used to prevent boundary blow-up still follows by a moving-planes technique as in \cite[Lemma 3.2]{dFdOR}, provided again some decaying conditions near $\dOmega$ are satisfied by the nonlinearity. Here, the Green function estimates needed in the Dirichlet case are replaced by the maximum principle for elliptic systems in small domains, see \cite{dF_monot}. The last significant element in order to retrace the blow-up method in the interior of the domain, is the counterpart of Lemma \ref{Greenspropetry}, in particular the estimates of the derivatives of the Green function $G_{NAV}$ of the biharmonic operator with Navier boundary conditions. We postpone this subject to Appendix A.

This enables us to retrace all the arguments of Section \ref{SectionBlowup} and, as a consequence, we are lead to the following:
\begin{thm}\label{NAV_uniformbound}
	Let $\Omega\subset\R^4$ be a bounded convex $C^{2,\gamma}$ domain and $h$ satisfy assumptions (\textit{H1})-(\textit{H2}).
	Furthermore, assume there exist $\bar r,\bar\delta>0$ such that  
	\begin{enumerate}[H3a\textquotesingle)]
		\item 
		$h(x,\cdot)\in C^1(\Omega_{\bar r})$  with $\frac{\partial h}{\partial t}(x,t)\geq 0$ in $\Omega_{\bar r}\times\R^+$
	\end{enumerate}
	and (H3b) holds. Then, there exists $C>0$ such that $\|u\|_{L^\infty(\Omega)}\leq C$ for all weak solutions of \eqref{NAVf}. 
\end{thm}

We point out that Theorem \ref{NAV_uniformbound} extends a previous result \cite[Corollary 2.3]{LinWei} established for problem \eqref{NAVf} for the special case of $h(x,u)=e^u$.

\section{Existence results}\label{SectionEXISTENCE}

The a-priori bound for solutions of \eqref{DIRf} obtained in the previous sections is essential to apply topological methods and infer the existence of positive solutions. We follow the standard approach which has been widely applied in the literature (see for instance \cite{dFLN,Soranzo,dFdOR2}), and which relies on a well-known result firstly due to Krasnosel'skii, which may be equivalently stated as follows (see \cite[Theorem 3.1]{dF} and the subsequent results).
\begin{lem}\label{Krasno}
	Let $X$ be a Banach space and $K\subset X$ a cone, which induces a partial order in $X$ defined as $x\leq y$ if and only if $y-x\in K$. Moreover, let $\Phi :K\to K$ be a compact map with $\Phi (0)=0$ and suppose there exist $0<r<R$ and $\tau>0$ such that:
	\begin{enumerate}
		\item there exists a bounded linear operator $A:X\to X$ such that $A(K)\subset K$ with spectral radius $r(A)<1$ and such that $\Phi(x)\leq Ax$ for all $x\in K$ with $\|x\|=r;$
		\item there exists $\Psi:K\times[0,+\infty)\to K$ such that
		\begin{enumerate}
			\item $\Psi(x,0)=\Phi(x)$;
			\item $\Psi(x,t)\neq x$ for all $t\geq 0$ and $\|x\|=R$;
			\item $\Psi(x,t)\neq x$ for all $t\geq\tau$ and $\|x\|\leq R$.
		\end{enumerate}
	\end{enumerate}
	Then $\Phi$ has a fixed point $\overline x\in K$ with $r<\|\overline x\|<R$.
\end{lem}

\begin{proof}[Proof of Theorem \ref{Existence}]
	Let $X=C(\overline\B)$ and $K:=\{f\in C(\overline\B)\,|\,f\geq 0\}$ be the closed cone of nonnegative functions, which induces the standard pointwise order on $C(\overline\B)$. Moreover, let $T:=(\Delta^2)^{-1}$ be the inverse of the bilaplace operator with Dirichlet boundary conditions, that is, $Tg=w$ if and only if $w$ solves
	\begin{equation*}
	\begin{cases}
	\Delta^2w=g\quad&\mbox{in } \B,\\
	w=w_n=0\quad&\mbox{on }\dB.
	\end{cases}
	\end{equation*}
	Then $T:C(\overline\B)\to C(\overline\B)$ is a linear compact and positive operator. Defining now $\Phi:=T\circ h(x,\cdot)$, then $\Phi(K)\subseteq K$ by maximum principle and $\Phi$ is a bounded compact operator, by the continuity of $h$. By \eqref{HP}, there exist $\alpha\in (0,\tilde\lambda_1)$ and $t_0>0$ such that $h(x,t)\leq\alpha t$ for any $x\in\B$, $t\in(0,t_0)$. Hence, defining $A:=\alpha T$, then:
	$$\Phi(u)=T(h(x,u))\leq T(\alpha u)=A(u),$$
	for any $u\in C(\overline\B)$ with $\|u\|_{C(\overline\B)}\le t_0$.
	Moreover $A(0)=0$ by definition, $A(K)\subset K$ and
	$$r(A)=\alpha\max\{\lambda\,|\,\lambda\mbox{ is an eigenvalue of }A^{-1}\}=\frac{\alpha}{\tilde\lambda_1}<1,$$
	so condition \textit{1} of Lemma \ref{Krasno} is verified.\\
	\indent Let us now define $\Psi(u,t):=T(h(x,u+t))$, for any $t\geq0$. It is clear that $\Psi(u,0)=\Phi(u)$. We are thus led to study the following family of problems:
	\begin{equation}\tag{P$_t$}\label{Pt}
	\begin{cases}
	\Delta^2u=h(x,u+t)\quad&\mbox{in }\B,\\
	u=u_n=0\quad&\mbox{on }\dB.
	\end{cases}
	\end{equation}
	With the same steps as in the proof of Lemma \ref{lemmaDir1-B}, we can get
	\begin{equation}\label{lemma1Dirindipt}
	\int_\B h(x,u+t)\tilde\varphi_1 dx\leq C,
	\end{equation}
	where the constant $C$ is independent of $t$; this in turns implies
	\begin{equation*}
	\int_\B a(x)(u+t)\tilde\varphi_1 dx\leq C.
	\end{equation*}
	Since $u$, $a$ and $\tilde\varphi_1$ are positive, we necessarily find that $t$ is bounded. This means that there are no (positive) solutions of problem \eqref{Pt} when $t>\tilde T$ for some $\tilde T>0$ depending on $\Omega$ and $h$, so condition (\textit{2c}) is fulfilled. Hence, we can restrict to $t\in[0,\tilde T]$ and prove the uniform a-priori bound for these values of the parameter $t$.
	In fact, in view of \eqref{lemma1Dirindipt}, firstly one can repeat the same steps of the proofs in Sections \ref{SectionNaPR} and produce the same uniform estimates of Proposition \ref{NEWSection1}. In particular, we get
	\begin{equation*}
	\int_\B h(x,u+t)dx\leq \Lambda(h,\tilde T),
	\end{equation*}
	uniformly with respect to $t\in[0,\tilde T]$. This is sufficient to guarantee that the contradictory argument of Section \ref{SectionBlowup} can be retraced for solutions of problem \eqref{Pt} with only minor adaptations, finding an a-priori bound which depends only on $h$ and $\tilde T$. Rephrased, this means that there are no solutions of problem \eqref{Pt} for any $t\in[0,\tilde T]$ with $\|u\|\geq R$ for some $R=R(\tilde T)>0$ (and thus also for any $t\geq 0$). Hence, condition (\textit{2b}) is verified.\\
	\noindent Since all assumptions of Lemma \ref{Krasno} are then ensured, the existence of a positive solution of problem \eqref{DIRf} follows. Regularity of such solutions is given by standard elliptic arguments.
\end{proof}

\begin{remark}
	All the arguments used in the proof of Theorem \ref{Existence} apply also for the polyharmonic context and a higher-order analogue of Theorem \ref{Existence} holds.
\end{remark}

\begin{remark}
	Notice that the assumption \eqref{HP} which was imposed to obtain condition \textit{1} of Lemma \ref{Krasno}, is also necessary in same cases. Indeed, if one considers $h(x,s)=\lambda sf(s)$, with $f(s)\geq 1$ for any $s\geq0$ (for instance, $f(s)=e^{s^\alpha}$ with $\alpha\in(0,1]$), then a simple calculation shows that if $\lambda>\tilde\lambda_1$, we have no positive solutions for problem \eqref{DIRf} in any domain. Indeed, by integration by parts, one gets
	$$\tilde\lambda_1\int_\Omega u\tilde\varphi_1=\int_\Omega\Delta u\Delta\tilde\varphi_1=\lambda\int_\Omega uf(u)\tilde\varphi_1\geq\lambda\int_\Omega u\tilde\varphi_1,$$
	which implies $\lambda\leq\tilde\lambda_1$.
	
\end{remark}

\section{A counterexample}\label{Sezcounterex}
	In this section we provide an example of a subcritical nonlinearity, in the sense of the Adams inequality \cite{Adams}, which does not satisfy assumption (\textit{A3}) and for which there exists an unbounded solution of the problem \eqref{DIRf}.\\
	For the second-order case, a counterexample may be found in \cite[\S II.2, Example 2]{BM}. However, its straightforward modification to our context does not work, since two main difficulties appear: first, boundary conditions are not all satisfied; secondly, the potential $a(\cdot)$ has to be positive, unlike there. Here we show how to overcome them in the fourth-order case. First, we correct the boundary values with the help of a biharmonic function; this will also need an intermediate growth between $log$ and $log\,log$ in the definition of our function in order to compensate some additional terms appearing in its behaviour near 0. Then, a detailed analysis is provided to show the positivity of $a$ in a whole ball around 0.
	\vskip0.2truecm
	In the following we denote  $l(r):=\log \frac{1}{r}$. Let us fix  $\alpha\in (1,2)$ and consider the family of functions 
	\[
	f_\beta(t)= t+ \beta t^{\gamma}+ \delta \log t,
	\]
	where $\beta>0$ is a positive parameter that will be chosen later, $\gamma = \frac{\alpha-1}{\alpha}\in (0,1)$ and $\delta = \frac{1}{4\alpha} - \frac{1}{2}<0$. Note that $f_\beta(t) \ge t+ \delta \log t $, hence there exists $\rho_0>0$ such that $f_\beta(l(r))>0$ for any  $r\in [0,\rho_0]$ and $\beta\ge 0$. Then, for any $x\in B_{\rho_0}(0)$ we define 
	$$
	u_\beta(r):= f_\beta(l(r))^\frac{1}{\alpha}, \quad r = |x|.
	$$
	\begin{lem}\label{Bilaplacian}
		For any choice of $\alpha$ and $\beta$, we have 
		\[\begin{split} 
		\Delta^2 u_\beta & = \frac{1}{r^4} \sum_{i=1}^4 i! \binom{\frac{1}{\alpha}}{i} u_\beta(r)^{1-i \alpha} h_i(r),
		\end{split}
		\]
		where 
		\[
		\begin{split}
		h_1(r)&=f''''_\beta(l(r))-4 f_\beta ''(l(r)), \\
		h_2(r)&= 4f_\beta '''(l(r))f_\beta'(l(r))-4 f_\beta'(l(r))^2 + 3 f_\beta''(l(r))^2, \\
		h_3(r)&=6 f_\beta '(l(r))^2 f_\beta''(l(r)), \\
		h_4(4)&=f_\beta '(l(r))^4.
		\end{split}
		\]
	\end{lem}
	\begin{proof}
		The proof is straightforward using the radial representation of the laplacian.
	\end{proof}

	\begin{lem}\label{Lemma2}
		There exists $\rho_1\in (0,\rho_0)$ such that for any $\rho \in (0,\rho_1)$ we can find  $A(\rho)<0$, $B(\rho)>0$, $\beta(\rho)>0$ such that
		\begin{equation}\label{parameters}
		\begin{cases}
		A(\rho)+B(\rho)\rho^2 = -u_{\beta(\rho)}(\rho),\\
		2B(\rho) \rho =-u_{\beta(\rho)}'(\rho),\\
		\beta(\rho)=-\alpha A(\rho).
		\end{cases}
		\end{equation}
		Moreover, there exist $C>0$ such that $\frac{1}{C}l(\rho)^\frac{1}{\alpha}<\beta(\rho)\le C l(\rho)^\frac{1}{\alpha}$ for any $\rho \in (0,\rho_1)$.
	\end{lem}
	\begin{proof}
		We define
		$$
		F(x,y,t)=\bigg( x +y t +(1-\alpha x-\delta t \log t)^\frac{1}{\alpha}, 2y - \frac{1}{\alpha}(1-\alpha x-\delta t \log t)^{\frac{1}{\alpha}-1}(1-(\alpha-1)x + \delta t)\bigg).
		$$
		Since $\alpha>1$, there exists a unique $x_0<0$ such that $(-x_0)^\alpha =1-\alpha x_0$ (actually, we also have $x_0<-1$). If $y_0 =-\frac{x_0+(-x_0)^{2-\alpha}}{2\alpha}$, then $F(x_0,y_0,0)=0.$  Since $x_0\neq -1$, one can easily verify that $\frac{\partial F}{\partial (x,y)}(x_0,y_0,0)$ is invertible. By the implicit function theorem, we can find $T_1>0$ such that for any $t\in (0,T_1)$ there exist $x(t)<0$ and $y(t)>0$ such that 
		$F(x(t),y(t),t)=0$. Let $\rho_1= e^{-\frac{1}{T_1}}$, then for any $\rho\in (0,\rho_1)$ the functions $A(\rho):= x(l(\rho)^{-1}) l(\rho)^\frac{1}{\alpha}$, $B(\rho):= \rho^{-2} y(l(\rho)^{-1})l(\rho)^{\frac{1}{\alpha}-1}$, $\beta(\rho)=-\alpha A(\rho)$ satisfy \eqref{parameters}. Moreover, by construction we have 
		$$
		\beta(\rho) = -\alpha x(l(\rho)^{-1}) l(\rho)^\frac{1}{\alpha} = -\alpha x_0 l(\rho)^\frac{1}{\alpha} (1+o(1))
		$$
		as $\rho \to 0$. 
	\end{proof}
	
	Next we consider the function $u_\beta$ with $\beta = \beta(\rho)$ and a sufficiently small $\rho$. 
	
	\begin{lem}
		For any small  $\rho\in (0,\rho_1)$ we have $\Delta^2 u_{\beta(\rho)} >0$ in $B_\rho(0)\setminus \{0\}$. 
	\end{lem}
	\begin{proof} 
		For any fixed $\rho\in (0,\rho_1)$, let $h_i(r)$, $i=1,...,4$ be as in the Lemma \ref{Bilaplacian} with $\beta = \beta(\rho)$.  Note that, if $\rho$ is sufficiently small,  we have 
		\begin{equation}\label{f'below}
		f'_{\beta(\rho)}(l(r))= 1+ \beta(\rho)\gamma l(r)^{\gamma-1}  + \delta l(r)^{-1} \ge 1+\delta l(r)^{-1}\ge \frac{1}{2},
		\end{equation}
		and 
		\begin{equation}\label{f'above}
		f'_{\beta(\rho)}(l(r)) = 1+ \beta(\rho)\gamma l(r)^{\gamma-1}  + \delta l(r)^{-1} \le 1+ \gamma C \frac{l(r)^{\gamma-1}}{l(\rho)^{\gamma-1}}\le 1+\gamma C,
		\end{equation}
		for any $r\in (0,\rho)$. Similarly, we get 
		\begin{equation}\label{f''}
		f''_{\beta(\rho)}(l(r))= \beta (\rho)\gamma(\gamma-1) l(r)^{\gamma-2} - \frac{\delta}{l(r)^2} \le -\frac{\delta}{l(r)^2}.
		\end{equation}
		and 
		\begin{equation}\label{f^(i)}
		|f^{(i)}_{\beta(\rho)}(l(r))| = \big|\beta(\rho)i! \binom{\gamma}{i} l(r)^{\gamma-i}  +  \frac{(-1)^{i-1}}{(i-1)!} l(r)^{-i}\big| \,\leq\, Cl(\rho)^{1-\gamma} l(r)^{\gamma-i} .
		\end{equation}
		Since $\delta<0$, using \eqref{f''} and  \eqref{f^(i)} we get 
		\begin{equation}\label{h1_below}
		h_1(r)\ge  \frac{4\delta}{l(r)^2}\left(1+O(l(\rho)^{-1})\right)\ge  - \frac{C}{l(\rho)l(r)},
		\end{equation}
		for some $C>0$. Similarly, \eqref{f'below}, \eqref{f'above} and \eqref{f^(i)} yield
		\begin{equation}\label{h2}
		h_2(r) \le -1 + O(l(\rho)^{-2}).
		\end{equation}
		Finally, by \eqref{f'below}-\eqref{f^(i)} we get 
		\begin{equation}\label{h34}
		h_3(r)= O(l(\rho)^{-1})\quad \text{and} \quad h_4(r)=O(1).
		\end{equation}
		Furthermore we have   
		\begin{equation}\label{bounds_u}
		\frac{1}{2} l(r)^\frac{1}{\alpha} \le (l(r) + \delta\log l(r))^\frac{1}{\alpha} \le u_{\beta(\rho)}(r)\le ((1+ C)l(\rho) + \delta\log l(r))^\frac{1}{\alpha} \le C l(r)^\frac{1}{\alpha}.
		\end{equation}
		Then, using \eqref{h1_below} and \eqref{bounds_u} we get
		\[\begin{split}
		\binom{\frac{1}{\alpha}}{1} u_{\beta(\rho)}(r)^{1-\alpha} h_1(r) &=  \frac{1}{\alpha} u_{\beta(\rho)}(r)^{1-\alpha} h_1(r) \ge -\frac{C}{\alpha} \frac{u_{\beta(\rho)}(r)^{1-\alpha}}{l(r)l(\rho)} \\
		&\ge -\frac{C}{\alpha} \frac{u_{\beta(\rho)}(r)^{1-2\alpha}u_{\beta(\rho)}(r)^\alpha}{l(r)l(\rho)}
		\ge -C \frac{u_{\beta(\rho)}(r)^{1-2\alpha}}{l(\rho)}.
		\end{split}
		\]
		Thanks to \eqref{h2} we find
		\[\begin{split}
		2\binom{\frac{1}{\alpha}}{2} u_{\beta(\rho)}(r)^{1-2 \alpha} h_2(r) &= \frac{1-\alpha}{\alpha^2} u_{\beta(\rho)}(r)^{1-2\alpha} h_2(r)\ge \frac{1-\alpha}{\alpha^2} u_{\beta(\rho)}(r)^{1-2\alpha} (-1 + O(l(\rho)^{-2})).\\
		\end{split}
		\]
		Similarly, \eqref{h34} yields
		\[\begin{split}
		3!\binom{\frac{1}{\alpha}}{3} u_{\beta(\rho)}(r)^{1-3 \alpha} h_3(r) &= \frac{(1-\alpha)(1-2\alpha)}{\alpha^3}  u_{\beta(\rho)}(r)^{1-3\alpha} h_3(r) = u_{\beta(\rho)}(r)^{1-2\alpha} O(l(\rho)^{-2})
		\end{split}
		\]
		and 
		\[\begin{split}
		4!\binom{\frac{1}{\alpha}}{4} u_{\beta(\rho)}(r)^{1-4 \alpha} h_4(r) &= \frac{(1-\alpha)(1-2\alpha)(1-3\alpha)}{\alpha^4}  u_{\beta(\rho)}(r)^{1-4\alpha} h_4(r) = u_{\beta(\rho)}(r)^{1-2\alpha} O(l(\rho)^{-2}). 
		\end{split}
		\]
		In conclusion we have 
		$$
		\Delta^2 u_{\beta(\rho)} \ge \frac{ u_{\beta(\rho)}(r)^{1-2\alpha}}{r^4}  \left(\frac{\alpha-1}{\alpha^2} + O(l(r)^{-1}) \right)\ge 0
		$$
		for any $r\in (0,\rho)$ if $\rho$ is chosen sufficiently small. 
	\end{proof}
	
	Now, we fix $\rho $ as in the previous lemma and we let 
	$$
	w(r):= 4^\frac{1}{\alpha} \left(u_{\beta(\rho)}(r)+A(\rho) + B(\rho)r^2\right).
	$$
	Clearly, we have that $w\in C^\infty(\overline {B_\rho(0)} \setminus\{0\}) \cap W^{3,p}(B_\rho(0))$ for any $1<p<\frac{4}{3}$ and that $ \Delta^2 w = 4^\frac{1}{\alpha} \Delta^2 u_{\beta(\rho)}>0$. Using Lemma \ref{Bilaplacian} one can easily check that 
	\[
	\Delta^2 u_{\beta(\rho)} = \frac{4(\alpha-1)}{\alpha^2r^4} l(r)^{\frac{1}{\alpha}-2}(1+o(1))
	\]
	as $r\to 0$, and in particular $\Delta^2 w \in L^1(B_\rho(0))$. Moreover, for $r=\rho$ by Lemma \ref{Lemma2} we infer that $w(\rho)=w'(\rho)=0$. Hence, $w$ can be represented by Green's formula and thus, by the positivity of the Green function, we get that $w>0$ in $B_\rho(0)$. Therefore, $w$ is a distributional solution of 
	$$
	\begin{cases}
	\Delta^2 w = a(x) e^{w^{\alpha}} & \; \text{ in }B_\rho(0), \\
	w = w_n =  0 & \; \text{ on } \partial B_\rho(0),
	\end{cases}
	$$
	where $a:=  4^\frac{1}{\alpha} e^{-w^\alpha } \Delta^2 u_{\beta(\rho)} $. By construction, $a\in C^\infty(\overline{B_\rho(0)}\setminus\{0\})$. We now prove that $a \in L^\infty(B_\rho(0))$ and, more precisely, that $a$ extends to a positive continuous function on $\overline{B_\rho(0)}$. First, note that by \eqref{bounds_u} and Lemma \ref{Lemma2} we have
	\[
	\begin{split}
	w^\alpha (r)= 4(u_{\beta(\rho)}(r)+ A +O(r^2))^\alpha  &= 4u_{\beta(\rho)}(r)^\alpha \bigg(1+ \frac{A}{u_{\beta(\rho)}(r)} + O(r^2l(r)^{-\frac{1}{\alpha}}) \bigg)^\alpha \\
	& = 4 u_{\beta(\rho)}^\alpha(r) + 4\alpha A u_{\beta(\rho)}^{\alpha-1}(r) + o(1)\\
	& = 4l(r)+4\beta(\rho)l(r)^\gamma + 4\delta\log l(r) -4 \beta(\rho) u_{\beta(\rho)}^{\alpha-1} (r)+o(1) \\
	&=  4l(r)+4\beta(\rho)l(r)^\gamma +4\delta\log l(r)-4\beta(\rho) l(r)^{\frac{\alpha-1}{\alpha}} + o(1),
	\end{split}
	\] 
	as $r\to 0$. Recalling now that $\gamma = \frac{\alpha-1}{\alpha}$, we get 
	\begin{equation}\label{w_in0}
	w^{\alpha} = 4l(r) + 4\delta\log l(r) +o(1)
	\end{equation}
	as $r\to 0$. Our claim is then proved as
	\[
	\begin{split}
	a(x) = 4^\frac{1}{\alpha} \Delta^2 u_{\beta(\rho)} e^{-w^{\alpha}}& = 4^{\frac{1}{\alpha}+1}\frac{\alpha-1}{\alpha^2 r^4} l(r)^{\frac{1}{\alpha}-2} r^4 l(r)^{-4\delta} (1+o(1))\\
	& = 4^{\frac{1}{\alpha}+1}\frac{\alpha-1}{\alpha^2} +o(1),
	\end{split}
	\]
	where we used that $4\delta=\frac{1}{\alpha}-2$. However, by \eqref{w_in0} one easily sees that $w\not\in L^\infty(B_\rho(0))$.

\section{An open problem}\label{SectionOP}

\indent We completely proved the uniform a-priori bound for solutions of problem \eqref{DIRf} only if the domain is a ball. However, as investigated in Section \ref{SectionDomains}, the blow-up technique may be applied independently of our domain, provided the energy control $\int_\Omega h(x,u)dx\leq\Lambda$ holds. It would be thus very interesting to find an alternative method to prove such estimate in a general domain. In particular, if the domain is still positivity preserving for $(-\Delta)^m$, we note that it would be sufficient to replace the moving-planes argument of Lemma \ref{MovingPlanes}. We believe that this should be possible in the context of small smooth deformations of the ball, which are the only known examples of positivity preserving domains in dimension $N\geq3$ (see \cite{GR}).

We note also that in the second-order context such $L^1$ bound has been recently proved in \cite{Batt} for simply connected planar domains. However, the techniques therein are peculiar to the two-dimensional setting, so they are not straightforwardly generalisable to the higher-order setting.

\section*{Appendix A: Green function estimates for Navier boundary conditions}

We are here interested in proving a counterpart of Lemma \ref{Greenspropetry} in the case of Navier boundary conditions. Firstly, decoupling \eqref{NAVf} as a system of second-order equations, one easily finds an explicit formula for the Green function $G_{NAV}$ and its derivatives,
\begin{equation}\label{GreenNav_i_def}
\nabla^i_xG_{NAV}(x,z)=\intOmega\nabla^i_x G_{-\Delta}(x,y) G_{-\Delta}(y,z)dy,
\end{equation}
in terms of $G_{-\Delta}$, the Green function for the Laplace operator with Dirichlet boundary conditions. In the case $i=0$, an estimate of the form \eqref{Greenspropetry_0} may be found in \cite{GS_nav}, see also \cite[Proposition 4.13]{GGS}. Here, inspired by some arguments therein and taking into account the respective estimates for $G_{-\Delta}$ (see \cite{Widman,Sweers}), we give a proof of the estimates for the derivatives of $G_{NAV}$. We remark that such result may also follow by \cite{Kraso} as a particular case of very general theory on kernels of elliptic operators; however, here we give a easier proof which does not require strong regularity assumptions on the boundary.
\begin{propA}
	Let $\Omega\subset\R^4$ be a bounded domain of class $C^{1,1}$ and let $G_{NAV}(x,y)$ be the Green function in $\Omega$ of the biharmonic operator subjected to Navier boundary conditions. There exists $C>0$ such that for all $x,y\in\Omega$, $x\neq y$, we have that
	\begin{equation*}
	|G_{NAV}(x,y)|\leq C\log\bigg(1+\dfrac{d_\Omega(x)d_\Omega(y)}{|x-y|^2}\bigg),
	\end{equation*}
	\begin{equation}\label{GreenNav_i}
	|\nabla_x^iG_{NAV}(x,y)|\leq \frac{C}{|x-y|^i}.
	\end{equation}
	with the convention that $\nabla^2f:=\Delta f$ and $\nabla^3f:=\nabla(\Delta f)$.
\end{propA}

\begin{proof}
	We start proving \eqref{GreenNav_i} for $i=1$. Indeed, the estimates for $i=2,3$ easily follow by means of the decomposition into coupled systems and the estimates for $G_{-\Delta}$. In the sequel, we strictly follow \cite{GS_nav}. Fix $x,z\in\Omega$ and define $\mathcal{O}_x:=B_{\frac23|x-z|}(x)\cap\Omega$, similarly $\mathcal{O}_z:=B_{\frac23|x-z|}(z)\cap\Omega$ and let then $\mathcal{R}:=\Omega\setminus(\mathcal{O}_x\cup\mathcal{O}_z)$.
	Applying the pointwise estimates for $G_{-\Delta}$ in dimension $4$ (see \cite[\S\,3.2]{Sweers}) and \eqref{GreenNav_i_def} we get
	\begin{equation*}
	|\nabla^i_xG_{NAV}(x,z)|\preceq\intOmega \frac{ \big(1\wedge\tfrac{d(y)}{|x-y|}\big)\big(1\wedge\tfrac{d(y)d(z)}{|y-z|^2}\big)}{|x-y|^3|y-z|^2}dy=:\intOmega Q(x,y,z)dy.
	\end{equation*}
	It is easy to show that on $\Ox$ we have $|y-z|\sim|x-z|$ and on $\Oz$ there holds $|x-y|\sim|x-z|$. Hence,
	 \begin{equation}\label{Ox}
	 \begin{split}
		 \int_{\Ox} Q(x,y,z)dy &\preceq\int_{\Ox}|x-y|^{-3}|y-z|^{-2}dy\preceq|x-z|^{-2}\int_{\Ox}|x-y|^{-3}dy\\
		 &\preceq|x-z|^{-2}\int_0^{\frac23|x-z|}d\rho\preceq\frac{1}{|x-z|}.
	 \end{split}
	 \end{equation}
	 and an analogous estimate for $Q(x,\cdot,z)$ on $\Oz$ holds.\\
	 On $\OR$ there holds $|y-z|\sim|x-y|$. In fact, on one hand $|y-z|\geq\frac23|x-y|$, so $|x-y|\leq|x-z|+|z-y|\leq\frac52|y-z|$, and, on the other, $|y-x|\geq\frac23|x-z|$, implying $|y-z|\leq|x-y|$. Furthermore, the following relation holds (see \cite[Lemma 3.2]{GS_nav2} as well as \cite[Lemma 4.5]{GGS}):
	 \begin{equation*}
	 \bigg(1\wedge\dfrac{d(y)d(z)}{|y-z|^2}\bigg)\sim\bigg(1\wedge\dfrac{d(y)}{|y-z|}\bigg)\bigg(1\wedge\dfrac{d(z)}{|y-z|}\bigg)\leq 1\wedge\dfrac{d(z)}{|y-z|}.
	 \end{equation*}
	 As a result, we get
	 \begin{equation}\label{OR_1}
	 \begin{split}
	 \int_{\OR} Q(x,y,z)dy &\preceq\int_{\OR}|x-y|^{-5}\bigg(1\wedge\dfrac{d(y)}{|x-y|}\bigg)\bigg(1\wedge\dfrac{d(z)}{|y-z|}\bigg)\\
	 &\preceq\int_{\Ox^c}|x-y|^{-5}\bigg(1\wedge\dfrac{1}{|x-y|}\bigg)\bigg(1\wedge\dfrac{d(z)}{|x-y|}\bigg).
	 \end{split}
	 \end{equation}
	 To estimate \eqref{OR_1} we distinguish two cases depending on the reciprocal distance of $x$ and $z$ compared with their distance from the boundary. We denote by $D$ a sufficiently large radius so that $\OR\subset B_D(x)\setminus\Ox$.\\
	 \textbf{Case $|x-z|^2\leq d(z)$}. If so, one has $\frac23|x-z|\leq\sqrt{d(z)}$ and thus we continue \eqref{OR_1} as follows:
	 \begin{equation}\label{OR_2}
	 \begin{split}
	 \int_{\OR} Q(x,y,z)dy&\preceq\int_{\frac23|x-z|}^D\rho^{-5}\bigg(1\wedge\dfrac{1}{\rho}\bigg)\bigg(1\wedge\dfrac{d(z)}{\rho}\bigg)\rho^3d\rho\\
	 &\preceq\int_{\frac23|x-z|}^{\sqrt{d(z)}}\rho^{-2}d\rho+d(z)\int_{\sqrt{d(z)}}^D\rho^{-4}d\rho\\
	 &\preceq\frac 1{|x-z|}+\dfrac1{d(z)}\preceq\frac 1{|x-z|}.
	 \end{split}
	 \end{equation}
	 \textbf{Case $|x-z|^2> d(z)$}. Then,
	 \begin{equation}\label{OR_3}
	 \int_{\OR} Q(x,y,z)dy\preceq d(z)\int_{\frac23|x-z|}^D\rho^{-4}d\rho\preceq \frac{d(z)}{|x-z|^3}<\frac 1{|x-z|}.
	 \end{equation}
	 The result is obtained combining estimates \eqref{Ox}, the analogue in $\Oz$, \eqref{OR_2} and \eqref{OR_3}.
	\end{proof}


\begin{thebibliography}{ABCDE}
	\bibitem[Ad]{Adams} Adams D. 
	{\em A sharp inequality of J. Moser for higher order derivatives}, Ann. of Math. (2) 128 (1988), no. 2, 385-398.

	
	\bibitem[ARS]{ARS} Adimurthi, Robert F., Struwe M.
	{\em Concentration phenomena for Liouville's equation in dimension four.} J. Eur. Math. Soc. (JEMS) 8 (2006), no. 2, 171-180. 
	
	\bibitem[Ba]{Batt} Battaglia L. {\em Uniform bounds for solutions to elliptic problems on simply connected planar domains.} arXiv:1809.05684
	
	\bibitem[BGW]{BGW} Berchio E., Gazzola F., Weth T.
	{\em Radial symmetry of positive solutions to nonlinear polyharmonic Dirichlet problems.} J. Reine Angew. Math. 620 (2008), 165-183.
	
	\bibitem[BM]{BM} Brezis H., Merle F.
	{\em Uniform estimates and blow-up behavior for solutions of $-\Delta u=V(x)e^u$ in two dimensions.}	Comm. Partial Differential Equations 16 (1991), no. 8-9, 1223-1253.
	
	
	\bibitem[BT]{BT} Brezis H., Turner R.E.L.
	{\em On a class of superlinear elliptic problems.} Comm. Partial Differential Equations 2 (1977), no. 6, 601-614. 
	
	\bibitem[CL]{ChenLi} Chen W.X., Li C.
	{\em A priori estimates for solutions to nonlinear elliptic equations.}	Arch. Rational Mech. Anal. 122 (1993), no. 2, 145-157. 
	
	\bibitem[CS]{CS} Cl\'{e}ment P., Sweers G.
	{\em Uniform anti-maximum principle for polyharmonic boundary value problems.} Proc. Amer. Math. Soc. 129 (2001), no. 2, 467-474. 
	
	\bibitem[DS]{DAS} Dall'Acqua A., Sweers G.
	{\em Estimates for Green function and Poisson kernels of higher-order Dirichlet boundary value problems.} J. Differential Equations 205 (2004), no. 2, 466-487.
	
	\bibitem[dF1]{dF} de Figueiredo D.G.
	{\em Positive Solutions of Semilinear Ellpitic Equations.} Differential equations (S\~{a}o Paulo, 1981), pp. 34-87, Lecture Notes in Math., 957, Springer, 1982.
	
	\bibitem[dF2]{dF_monot} de Figueiredo D.G.
	{\em Monotonicity and symmetry of solutions of elliptic systems in general domains.} NoDEA Nonlinear Differential Equations Appl. 1 (1994), no. 2, 119-123.
	
	\bibitem[dFLN]{dFLN} de Figueiredo D.G., Lions P.-L., Nussbaum R.D.
	{\em A priori estimates and existence of positive solutions of semilinear elliptic equations.}	J. Math. Pures Appl. (9) 61 (1982), no. 1, 41-63. 
	
	\bibitem[dFdOR]{dFdOR} de Figueiredo D., do \'{O} J.M., Ruf B.
	{\em Semilinear elliptic systems with exponential nonlinearities in two dimensions.} Adv. Nonlinear Stud. 6 (2006), no. 2, 199-213.
	
	\bibitem[dFdOR2]{dFdOR2} de Figueiredo D., do \'{O} J.M., Ruf B.
	{\em Non-variational elliptic systems in dimension two: a priori bounds and existence of positive solutions.} J. Fixed Point Theory Appl. 4 (2008), no. 1, 77-96.
	
	\bibitem[DST]{DST} Dur\'{a}n R. G., Sanmartino M., Toschi M.
	{\em On the existence of bounded solutions for a nonlinear elliptic system.} Ann. Mat. Pura Appl. (4) 191 (2012), no. 4, 771-782.
	
	\bibitem[FGW]{FGW} Ferrero A., Gazzola F., Weth T.
	{\em Positivity, symmetry and uniqueness for minimizers of second-order Sobolev inequalities.}	Ann. Mat. Pura Appl. (4) 186 (2007), no. 4, 565-578. 
	
	\bibitem[GGS]{GGS} Gazzola F., Grunau H.-C., Sweers G.
	{\em Polyharmonic boundary value problems.} Springer Lecture Notes in Mathematics n. 1991, 2010. xviii+423 pp.
	
	\bibitem[GiSp]{GS} Gidas B., Spruck J.
	{\em A priori bounds for positive solutions of nonlinear elliptic equations}, Comm. Partial Differential Equations 6 (1981), no. 8, 883-901.
	
	\bibitem[GR1]{GR} Grunau H.-C., Robert F.
	{\em Positivity and almost positivity of biharmonic Green's functions under Dirichlet boundary conditions.} Arch. Ration. Mech. Anal. 195 (2010), no. 3, 865-898.
	
	\bibitem[GR2]{GRholes} Grunau H.-C., Robert F.
	{\em Uniform estimates for polyharmonic Green functions in domains with small holes.} Recent trends in nonlinear partial differential equations II: Stationary problems, Contemp. Math. 595 (2013), 263-272.
	
	\bibitem[GS1]{GS_nav2} Grunau H.-C., Sweers G.
	{\em Positivity for equations involving polyharmonic operators with Dirichlet boundary conditions.} Math. Ann. 307 (1997), no. 4, 589-626. 
	
	\bibitem[GS2]{GS_pp} Grunau H.-C., Sweers G.
	{\em Positivity properties of elliptic boundary value problems of higher order.} Proceedings of the Second World Congress of Nonlinear Analysts, Part 8 (Athens, 1996). Nonlinear Anal. 30 (1997), no. 8, 5251-5258.
	
	\bibitem[GS3]{GS_nav} Grunau H.-C., Sweers G.
	{\em Sharp estimates for iterated Green functions.}	Proc. Roy. Soc. Edinburgh Sect. A 132 (2002), no. 1, 91-120.
	
	\bibitem[JY]{JY} Jevnikar A., Yang W.
	{\em Analytic aspects of the Tzitz\'{e}ica equation: blow-up analysis and existence results.}  Calc. Var. Partial Differential Equations 56 (2017), no. 2, Art. 43, 38 pp.
	
	\bibitem[Kr]{Kraso} Krasovski\u{\i} J.P.
	{\em Investigation of potentials associated with boundary problems for elliptic equations.} Izv. Akad. Nauk SSSR Ser. Mat. 31 (1967) 587-640 (in Russian). English transl., Math. USSR Izv. 1 (1967) 569-622.
	
	\bibitem[La]{Lakkis} Lakkis, O.
	{\em Existence of solutions for a class of semilinear polyharmonic equations with critical exponential growth.}
	Adv. Differential Equations 4 (1999), no. 6, 877-906. 
	
	\bibitem[LL]{LamLu} Lam N., Lu G.
	{\em Existence of nontrivial solutions to polyharmonic equations with subcritical and critical exponential growth.} Discrete Contin. Dyn. Syst. 32 (2012), no. 6, 2187-2205. 
	
	\bibitem[LRU]{LRU} Lorca S., Ruf B., Ubilla P.
	{\em A priori bounds for superlinear problems involving the N-Laplacian.} J. Differential Equations 246 (2009), no. 5, 2039-2054.
	
	\bibitem[L]{Lin} Lin C.-S.
	{\em A classification of solutions of a conformally invariant fourth order equation in $R^n$.} Comment. Math. Helv. 73 (1998), no. 2, 206-231.
	
	\bibitem[LW]{LinWei} Lin C.-S., Wei J.-C.
	{\em Locating the peaks of solutions via the maximum principle. II. A local version of the method of moving planes.} Comm. Pure Appl. Math. 56 (2003), no. 6, 784-809.	

	\bibitem[Ma1]{Mart} Martinazzi L.
	{\em Classification of solutions to the higher order Liouville's equation on $\R^{2m}$}. Math. Z. 263 (2009), no. 2, 307-329.
	
	\bibitem[Ma2]{Mart_cc} Martinazzi L.
	{\em Concentration-compactness phenomena in the higher order Liouville's equation}.	J. Funct. Anal. 256 (2009), no. 11, 3743-3771. 
	
	\bibitem[MP]{MP} Martinazzi L., Petrache M.
	{\em Asymptotics and quantization for a mean-field equation of higher order.} Comm. Partial Differential Equations 35 (2010), no. 3, 443-464.
	
	\bibitem[Mi]{Mit} Mitidieri E.
	{\em A Rellich type identity and applications.} Comm. Partial Differential Equations 18 (1993), no. 1-2.
	
	\bibitem[Os]{Osw} Oswald P.
	{\em On a priori estimates for positive solutions of a semilinear biharmonic equation in a ball.} 
	Comment. Math. Univ. Carolin. 26 (1985), no. 3, 565-577.
	
	\bibitem[Pa]{Pass} Passalacqua T.
	{\em Some applications of functional inequalities to semilinear elliptic equations} Ph.D. Thesis, Universit\`{a} degli Studi di Milano (2015).
	
	\bibitem[Po]{Ponce} Ponce A.C.
	{\em Elliptic PDEs, Measures and Capacities. From the Poisson equation to Nonlinear Thomas-Fermi problems}. EMS Tracts in Mathematics. European Mathematical Society (EMS), Z\"{u}rich, 2016. 463 pp.
	
	\bibitem[RW]{RW} Reichel W., Weth T.
	{\em A priori bounds and a Liouville theorem on a half-space for higher-order elliptic Dirichlet problems.}  Math. Z. 261 (2009), no. 4, 805-827.
	
	\bibitem[RoWe]{RobWei} Robert F., Wei J.-C.
	{\em Asymptotic behavior of a fourth order mean field equation with Dirichlet boundary condition.} Indiana Univ. Math. J. 57 (2008), no. 5, 2039-2060.
	
	\bibitem[Rom]{MIO_N} Romani, G.
	{\em A-priori bounds for a quasilinear problem in critical dimension}, arXiv:1802.05777.
	
	\bibitem[Sor]{Soranzo} Soranzo R.
	{\em A priori estimates and existence of positive solutions of a superlinear polyharmonic equation.} Dynam. Systems Appl. 3 (1994), no. 4, 465-487.
	
	\bibitem[Sou]{Soup} Souplet, P.
	{\em Optimal regularity conditions for elliptic problems via $\Lp_\delta$-spaces.} Duke Math. J. 127 (2005), no. 1, 175-192. 
	
	\bibitem[Sw]{Sweers} Sweers G.
	{\em Positivity for a strongly coupled elliptic system by Green function estimates.} J. Geom. Anal. 4 (1994), no. 1, 121-142.
	
	\bibitem[Wi]{Widman} Widman K.-O.
	{\em Inequalities for the Green function and boundary continuity of the gradient of solutions of elliptic differential equations} Math. Scand. 21, 1967, 17-37.

\end{thebibliography}
\end{document}